\documentclass[reqno,11pt]{amsart}

\usepackage{mathrsfs}
\usepackage{amsfonts}
\usepackage{amssymb}
\usepackage{amsmath}
\usepackage{graphicx}
\usepackage{amssymb,latexsym, amsmath}
\usepackage{amsfonts,amsthm}
\usepackage{amscd}
\usepackage{mathrsfs}
\usepackage{hyperref}
\usepackage{epstopdf}
\usepackage{verbatim}
\usepackage{stmaryrd}

\numberwithin{equation}{section}
\newcommand{\inpt}[1]{\langle #1 \rangle}

\theoremstyle{plain}
\newtheorem{theorem}{Theorem}[section]

\newtheorem{lemma}[theorem]{Lemma}

\theoremstyle{definition}
\newtheorem{definition}[theorem]{Definition}
\theoremstyle{remark}

\newcommand{\bl}{\label}

\newcommand{\om}{\omega}
\newcommand{\gw}{\Omega}
\newcommand{\ztw}{\theta_t \omega}

\newcommand{\ctw}{\theta_{-t} \omega}

\newcommand{\mbR}{\mathbb{R}^n}
\newcommand{\mbr}{\mathbb{R}}
\newcommand{\tup}{\textup}
\newcommand{\vp}{\varphi}
\newcommand{\gz}{\theta}
\newcommand{\gk}{\kappa}
\newcommand{\msd}{\mathscr{D}}
\newcommand{\msa}{\mathcal{A}}
\newcommand{\intr}{\int_{\mbR}}
\newcommand{\rxr}{\rho \left(\frac{|x|^2}{r^2}\right)}

\theoremstyle{definition}
\theoremstyle{remark}

\theoremstyle{definition}

\theoremstyle{remark}

\numberwithin{equation}{section}

\begin{document}

\title[Stochastic Wave Equation with Arbitrary Exponents]{\Small{Random Attractor for Stochastic Wave Equation \\ with
Arbitrary Exponent and Additive Noise on} $\mathbb{R}^n$}


\author{Hongyan Li}
\address{College of Management, Shanghai University of Engineering
Science, Shanghai 201620, China}
\curraddr{}
\email{hongyanlishu@163.com}
\thanks{The first author is supported by the National Natural Science Foundation of China (11101265) and
Shanghai Education Research and Innovation Key Project (14ZZ157)}

\author{Yuncheng You}
\address{Department of Mathematics and Statistics, University of
South Florida, Tampa, FL 33620, USA}
\curraddr{}
\email{you@mail.usf.edu}
\thanks{}

\subjclass[2010]{primary 35B40, 35B41, 35R60, 37L55; secondary 60H15}

\keywords{stochastic wave equations, 
random attractor, pullback asymptotic compactness, additive  noise, arbitrary nonlinear exponents, unbounded domain.}

\date{November 22, 2014}

\dedicatory{}

\begin{abstract}
Asymptotic random dynamics of weak solutions for a damped stochastic wave equation with the nonlinearity of arbitrarily large exponent and the additive noise on $\mathbb{R}^n$ is investigated. The existence of a pullback random attractor is proved in a parameter region with a breakthrough in proving the pullback asymptotic compactness of the cocycle with the quasi-trajectories defined on the integrable function space of arbitrary exponent and on the unbounded domain of arbitrary dimension.
\end{abstract}

\maketitle

\section{\large{\bf{Introduction}}}

In this paper, we study the asymptotic behavior of solutions of a damped stochastic wave equation with nonlinearity of arbitrarily large exponent and additive noise defined on the entire Euclidean space $\mathbb{R}^n$ of arbitry dimension,
\begin{eqnarray}\label{1}
    u_{tt} -\Delta u + \beta u_t +f(x, u)+\alpha u = g(x) +\varepsilon \, \sum^m_{j=1}h_j (x) \, \frac{dW_j}{dt},
\end{eqnarray}
for $t \geq \tau$, with the initial condition
\begin{eqnarray}\label{2}
    u(x, \tau)=u_{0}(x),   \ \ \ \ u_t(x, \tau)=u_{1}(x),
\end{eqnarray}
 where $\alpha, \beta$ and $\varepsilon$ are positive constants, $g$ and $h_j \, (j=1, 2,
 \cdots, m)$ are given functions defined on $\mathbb{R}^n$, $\beta u_t$ is a damping term,  $f(x, u)$ is an interaction function satisfying some
 dissipative  conditions, and $\{W_j\}^m_{j=1}$ are independent, two-sided, real-valued Wiener processes on a probability space which will be specified later.

The asymptotic dynamics of solutions for nonlinear wave equations and nonlinear hyperbolic evolutionary equations with linear or nonlinear or localized damping terms have been studied in last three decades by many authors, e.g. \cite{ACH}-\cite{Ba},
\cite{Caraballo}-\cite{ChepyVishik}, \cite{fei}-\cite{H91},
\cite{kha}-\cite{LiZhou}, \cite{MN, sellyou}, \cite{sun}-\cite{te},
\cite{WaZh, You}. The established results naturally focus on the existence of global attractors by showing the absorbing property and the asymptotic compactness of the solution semigroups for autonomous system \cite{Ba, GT, te, You} or the skew-product flow for nonautonomous system \cite{ChepyVishik02, ChepyVishik, LiZhou}.

For stochastic wave equations, the solution mapping defines a random
dynamical system or called a cocycle, which is defined on a state
space with a parametric base space. Pullback random attractor (which
will simply be called random attractor) is the appropriate object
for describing asymptotic random dynamics of such a random dynamical
systems \cite{BD}-\cite{BaLuWa}, \cite{Chow}-\cite{Fan06},
\cite{Fan, JonesWang13, Lv, Shen, WangZhouGu}. The topic of random
attractors for stochastic wave equations has been studied by several
authors \cite{Chow, Crauel, Fanm, Fan, Lv, Shen, Wang, WangZhouGu,
Yangduan}. In regard to stochastic nonlinear wave equations driven
by additive noise, the existence of the random attractor has been
established for bounded domains \cite{Crauel, Fan, Lv, Shen,
Yangduan} and with the critical exponents on the unbounded domain
$\mathbb{R}^3$ in \cite{Wang11}. However, the problem of random
attractors is open for the stochastic wave equations with
\emph{nonlinearity of arbitrarily large exponents} and \emph{on the
unbounded domain} $\mathbb{R}^n$ with arbitrary dimension $n$. This
is the topic as well as the main contribution in this work.

In case of high-order nonlinearity and high dimensional unbounded domain, the issue of pullback asymptotic compactness for a random dynamical system becomes much more difficult to handle due to not only the lack of compactness of the Sobolev embeddings but also the necessarily involved high-order integrable function spaces, in addition to the cumbersome effect caused by the additive noise. In this work we shall resolve this challenging problem and accomplish the proof of random attractor by means of

    1) the uniform estimates for absorbing property and norm-smallness of solutions

    \hspace{8pt} outside a large ball;

    2) the grouping esimates of the extended energy functional for the compactness

    \hspace{8pt} in the space $H^1 (\mbR) \times L^2 (\mbR)$;

    3) the convergence criterion of Vitali type for the function space $L^p (\mathbb{R}^n)$ with an

    \hspace{8pt} arbitrary exponent and no limitation on the space dimension $n$,

\noindent  to circumvent the crucial difficulties especially related to the phase space $(H^1 (\mbR)\cap L^p (\mbR)) \times L^2 (\mbR)$.


In Section 2, we recall basic concepts and results related to random attractors and random dynamical systems. We shall transform the stochastic wave equation to a pathwise random wave equation through  Ornstein-Uhlenbeck processes and define the associated cocycle. In Section 3, we shall conduct uniform estimates of solutions for the absorbing sets and the tail part. In Section 4, we shall establish the intricate asymptotic compactness of the cocycle with respect to the intersection of $L^p (\mathbb{R}^n)$ and the Hilbert energy space. In Section 5, the existence of a random attractor for this stochastic wave equation with unlimited growth rate and additive noise on the unbounded domain with unlimited dimension is finally proved.

In this paper, we shall use $\|\cdot\|$ and $\inpt{\cdot, \cdot}$ to denote the norm and inner product of $L^2(\mathbb{R}^n),$ respectively. The norm of $L^r(\mathbb{R}^n)$ or a Banach space $X$ will be denoted by $\|\cdot\|_r$ or $\|\cdot\|_X$. We use $c, C$ or $C_i$ to denote generic or specific positive constants.

\section{\large{\bf{Preliminaries and the Random Dynamical System}}}

Let $(\Omega,{\mathcal{F}}, P)$ be a probability space and $(X, \|\cdot\|_X)$ be a real 
Banach space whose Borel $\sigma$-algebra is denoted by ${\mathscr{B}}(X)$.

\begin{definition}  \label{D1}
Let a mapping $\theta_{t}: \mathbb{R}\times \Omega\rightarrow \Omega$ be $({\mathscr{B}}(\mathbb{R})\times {\mathcal{F}}, {\mathcal{F}})$-measurable such that
 $\theta_{0}$ is the identity on $\Omega$, $\theta_{t+s}=\theta_{t}\circ\theta_{s} $ for all $t, s\in \mathbb{R}, \omega \in \Omega$, and $P\theta_{t} = P$ for all $t\in \mathbb{R}$. Then $(\gw, \mathcal{F}, P, \{\theta_t\}_{t\in \mbr})$ is called a \emph{parametric dynamical system} (briefy PDS).
\end{definition}

\begin{definition}  \bl{D2}
A mapping $\Phi: \mathbb{R}^+\times \Omega \times X\rightarrow X$ is called a \emph{random dynamical system} on $X$ over $(\Omega,{\mathcal{F}}, P, \{\theta_{t}\}_{t\in \mathbb{R}})$, if $\Phi $ is $({\mathscr{B}}(\mathbb{R}^+)\times {\mathcal{F}}\times {\mathscr{B}}(X), {\mathscr{B}}(X))$-measurable and for all $\omega\in \Omega$ and $t, s\in \mathbb{R}^+$ the following conditions are satisfied:

(i) $\Phi(0, \omega, \cdot)$ is the identity on $X$.

(ii) $\Phi(t+s, \omega, \cdot)=\Phi(t, \theta_{s}\omega, \Phi(s, \omega, \cdot))$.

(iii) $\Phi(t, \omega, \cdot): X\rightarrow X$ is strongly continuous.

\noindent Such a random dynamical system can be denoted by $(\Phi, \theta)$ or simply by $\Phi$.
\end{definition}

\begin{definition} \bl{D3}
    A nonempty mapping $D(\om): \gw \to 2^X$ is called a \emph{random set} in $X$, if the mapping $\om \longmapsto dist_X (x,D(\om))$ is measurable with respect to $\mathcal{F}$ for any given $x \in X$.

    1) A \emph{bounded} random set $B(\om) \subset X$ means that there is a random variable $r (\om) \geq 0$ such that $\|B(\om)\| =\sup_{x \in B(\om)} \|x\| \leq r(\om),\, \om \in \gw$.

    2) A random set $S(\om) \subset X$ is called \emph{compact} (reps. \emph{precompact}) if for all $\om \in \gw$ the set $S(\om)$ is a compact (reps. precompact) set in $X$.

    3) If a bounded random set $B(\om)$ satisfies the condition that, for any constant $\gk > 0$,
$$
    \lim_{t \to \infty} e^{-\gk t} \|B(\gz_{-t} \om)\| = 0, \quad \om \in \gw,
$$
then it is called \emph{tempered} with respect to $\{\gz_t\}_{t\in \mathbb{R}}$. A random variable $R: (\gw, \mathcal{F}, P) \to (0, \infty)$ is called \emph{tempered} if
$$
    \lim_{t \to - \infty} \frac{1}{t}\, \log^+ R(\theta_t \om) = 0,  \quad \om \in \gw.
$$
\end{definition}

We shall denote by $\msd_X$ or simply $\msd$ the family of tempered random subsets of $X$, which is inclusion-closed and called a \emph{universe}.

\begin{definition} \bl{D4}
Let $\Phi$ be a random dynamical system on $X$ over the parametric dynamical system $(\Omega,{\mathcal{F}}, P, \{\theta_{t}\}_{t\in \mathbb{R}})$. Let $\msd$ be a given universe of tempered random subsets of $X$.

1) A random set $K = \{K( \omega)\}_{\omega \in \Omega}\in \msd$ is a $\msd$-\emph{pullback absorbing set} for $\Phi$ with respect to $\msd$, if for any $B \in \msd$ there exists $t_B (\om) \geq 0$ such that
    $$
        \Phi (t, \gz_{-t} \om, B(\gz_{-t} \om)) \subset K (\om), \; \;\; \tup{for all} \;\, t \geq t_B (\om), \;\, \om \in \gw.
    $$

2) $\Phi$ is called $\msd$-\emph{pullback asymptotically compact} if for any $\om\in\gw$,
$$
    \{\Phi (t_m, \, \gz_{-t_m} \om,\, x_{m})\}_{m=1}^\infty \; \, \tup{has a convergent subsequence in} \; X,
$$
whenever $t_m \to \infty$ and $x_m \in B(\gz_{-t_m} \om)$ for any given $B \in \msd$.
\end{definition}

\begin{definition} \bl{D5}
    A random set $\msa = \{\msa (\om)\}_{\om\in\gw} \in \msd$ is called a \emph{random attractor} for a random dynamical system $(\Phi, \theta)$ with the attraction basin $\msd$, if the following conditions are satisfied:

    (i) $\msa$ is a compact random set.

    (ii) $\msa$ is invariant, $\Phi (t, \om, \msa (\om)) = \msa (\gz_t \om)$, for all  $t \geq 0$ and $\om \in \gw$.

    (iii) $\msa$ pullback attracts every set $B \in \msd$ in the sense
$$
    \lim_{t \to \infty} dist_X (\vp (t, \gz_{-t} \om, B(\gz_{-t} \om)), \msa (\om)) = 0, \quad \om \in \gw,
$$
where $dist_X (\cdot, \cdot)$ is the Hausdorff semi-distance with respect to the $X$-norm.
\end{definition}

The following result on the existence of random attractor for a random dynamical system has been established in \cite{BaLuWa, Crauel, Wang}.

\begin{theorem}  \label{T1}
Let ${\msd}$ be a universe of nonempty tempered random subsets of a Banach space $X$ and let $\Phi$ be a random
dynamical system on $X$ over $(\Omega, {\mathcal{F}}, P, \{\theta_{t}\}_{t\in \mathbb{R}})$.
Suppose that $K = \{K(\omega)\}_{\omega \in \Omega}$ is a closed $\msd$-pullback absorbing set for $\Phi$ with respect to ${\msd}$ and $\Phi$ is
${\msd}$-pullback asymptotically compact in $X$. Then $\Phi$ has a unique random attractor
$\mathcal{A} = \{\mathcal{A}(\omega)\}_{\omega \in \Omega}$ in $X$ with the attraction basin $\msd$, which is given by
\begin{eqnarray} \label{21}
{\mathcal{A}}(\omega)=\bigcap_{\tau \geq 0}\, \overline{\bigcup_{t\geq \tau}\Phi(t, \theta_{-t}\omega, K( \theta_{-t}\omega))}, \quad \omega \in \Omega.
\end{eqnarray}
\end{theorem}



Now we formulate the problem \eqref{1}-\eqref{2}. Let $\xi=u_t+\delta u$, where $\delta$ is a
positive number to be determined. Then (\ref{1})-(\ref{2}) can be rewritten as
\begin{equation} \label{22} 
\begin{split}
    &u_t+\delta u=\xi,   \\
    \xi_{t}-\delta\xi+\delta^2 u + \alpha u -\Delta u &+\beta (\xi-\delta u)+f(x, u)=g(x) +\varepsilon\sum^m_{j=1}h_j \frac{dW_j}{dt},  \\
    u(x, \tau) =u_{0} (x), \quad   &\xi(x, \tau)=\xi_{0}(x) =u_{1}(x)+\delta u_{0}(x).
\end{split} 
\end{equation}
The autonomous wave equation driven by the stochastic perturbation has the nonautonomos nature as the parametric stochastic processes $W_j(t)$ evolve. Hence the initial condition can be given at any initial time $\tau \in \mathbb{R}$.

\textbf{Assumption I} \;
Assume that $g \in H^1(\mathbb{R}^n)$ and for each $j=1, 2, \cdots, m,$ the function $h_j\in
H^2(\mathbb{R}^n)\cap W^{2,p}(\mathbb{R}^n)$, where $p > 2$ is arbitrarily given.

\textbf{Assumption II}\;
Assume that the nonlinear term $f\in C^1(\mathbb{R}^n\times
\mathbb{R}, \mathbb{R})$ and its antidrivative $F(x, u)=\int^u_0 f(x, s)ds$ satisfy the
following conditions:
\begin{align}
    &|f(x, u)|\leq C_1 |u|^{p -1}+\phi_1(x), \quad \phi_1(x)\in H^1(\mathbb{R}^n), \label{33} \\[3pt]
    &f(x, u)u - C_2F(x, u)\geq \phi_2(x), \quad \phi_2(x)\in L^1(\mathbb{R}^n),  \label{34} \\[3pt]
    &F(x, u)\geq C_3 |u|^{p}-\phi_3(x), \quad \phi_3(x)\in L^1(\mathbb{R}^n),  \label{35}
\end{align}
where $C_1, C_2$ and $C_3$ are positive constants and the arbitrarily given $p > 2$ is the same as in Assumption I. 

Assume that $\{W_j\}^m_{j=1}$ are independent, two-sided, real-valued Wiener processes on the canonical probability space $(\Omega, \mathcal{F}, P)$, where
$$
    \Omega = \{\omega=(\omega_1, \omega_2, \cdots, \omega_m)\in C(\mathbb{R}, \mathbb{R}^m): \omega(0)=0\},
$$
$\mathcal{F}$ is the Borel $\sigma$-algebra induced by the compact-open topology of $\Omega$ and $P$ is the corresponding Wiener measure on
$(\Omega, \mathcal{F})$. We will identify $\omega$ with a sample path
$$
    \omega (t) = (\omega_1 (t), \omega_2 (t), \cdots, \omega_m (t)), \quad \ t \in \mathbb{R}.
$$
Define the time shift operartor $\theta_t$ by
$$
    (\theta_t\omega) (\cdot)=\omega(\cdot + t)-\omega(t), \quad \omega\in \Omega, \; t\in \mathbb{R}.
$$
Then $(\Omega, \mathcal{F}, P, (\theta_t)_{t\in \mathbb{R}})$ is a parametric dynamical system.

To define a random dynamical system for the problem formulated by (\ref{22}), we now convert the
stochastic wave equation with the additive noise to a deterministic equation with a time-dependent random parameter $\theta_t \omega$ and random initial data. Given $j=1, 2, \cdots, m,$  we introduce one-dimensional Ornstein-Uhlenbeck equation
\begin{eqnarray}\label{36}
dz_j+\delta z_jdt=dW_j(t), \quad t \in \mathbb{R}.
\end{eqnarray}
Here $\delta$ is the same constant as in \eqref{22}. A solution to (\ref{36}) is given by
\begin{equation*}
    \begin{split}
    z_j (t, \omega_j) &= \int_{-\infty}^t e^{-\delta (t-s)}\,dW_j (s, \om) = -\delta\int^0_{-\infty}e^{\delta s} (\theta_t\omega_j)(s)\,ds \\[3pt]
    &= z_j (0, \theta_t \omega_j) := z_j (\theta_t \omega_j),  \quad t\in \mathbb{R}.
    \end{split}
\end{equation*}
Note that the random variable $|z_j(\omega_j)|$ is tempered and $z_j (\theta_t \omega_j)$ is continuous in $t$. Therefore, it follows from Proposition 4.3.3 in \cite{Ar} that there exists a tempered function $r_0 (\omega)>0$ such that
$$ 
    \sum^m_{j=1}(|z_j(\omega_j)|^2+|z_j(\omega_j)|^p)\leq r_0 (\omega).
$$
For the positive constant $\sigma$ specified later in \eqref{sigma},  the random variable $r_0 (\omega)$ satisfies
\begin{eqnarray}\label{38}
    r_0 (\theta_t\omega)\leq e^{(\sigma / 2) |t|} r_0 (\omega), \quad t\in \mathbb{R}, \;\; a.s.
\end{eqnarray}
Then it follows from (\ref{38}) that
\begin{eqnarray}\label{39}
\sum^m_{j=1}(|z_j(\theta_t\omega_j)|^2+|z_j(\theta_t\omega_j)|^p)\leq e^{(\sigma/2) |t|} r_0 (\omega), \quad t\in \mathbb{R}, \quad a.s.
\end{eqnarray}
Write $z(\theta_t\omega)=\sum^m_{j=1}h_j z_j(\theta_t\omega_j)$. By (\ref{36}), the abstract-valued Ornstein-Uhlenbeck process $z(\theta_t\omega)$ satisfies the stochastic equation
\begin{equation} \label{zqn}
    dz+\delta z dt=\sum^m_{j=1}h_j dW_j.
\end{equation}

Make a transformation 
\begin{equation}  \label{vxt}
    v(x, t)=\xi(x, t)-\varepsilon z(\theta_t\omega).
\end{equation}
Then (\ref{22}) is converted to the following initial value problem:
\begin{equation}\label{310} 
\begin{split}
    &u_t=v+\varepsilon z(\theta_t\omega)-\delta u,\\[6pt]
    v_{t}-\delta v+(\delta^2+\alpha+A)u \, + &f(x, u) = g - \beta (v+\varepsilon z(\theta_t\omega)-\delta u) +2\varepsilon\delta z(\theta_t\omega), \\[6pt]
    u(x, \tau) =u_{0}(x), \ \  &v(x, \tau) =v_{0}(x) =u_{1}(x) +\delta u_{0}(x) -\varepsilon z(\theta_\tau \omega),
\end{split} 
\end{equation}
where $A=-\Delta$. Define the phase space
$$
    E= \left(H^1(\mathbb{R}^n) \cap L^p (\mathbb{R}^n)\right) \times L^2(\mathbb{R}^n)
$$
endowed with the norm
\begin{eqnarray}\label{311}
     \|(u, v)\|_{(H^1 \cap L^{p})\times L^2}= \left(\|\nabla u\|^2+\|u\|^2+\|v\|^2\right)^{\frac 12} + \|u\|_{L^p}, \;\; \textup{for} \ \ (u, v) \in E.
\end{eqnarray}

\begin{lemma} \label{Lwks}
    Under the Assumptions \textup{I} and \textup{II}, for every $\omega\in \Omega$ and any given $g_0 = (u_{0}, v_{0})\in E$, the initial value problem \eqref{310} has a unique global weak solution
$$
    (u(\cdot, \omega, \tau, u_{0}), v(\cdot, \omega, \tau, v_{0}))\in C([\tau, \infty), E).
$$
Moreover,

\textup{1)} The solution $(u(t, \omega, \tau, u_{0}), v(t, \omega, \tau, v_{0}))$ is $(\mathscr{B}(\mathbb{R}^+)\times \mathcal{F} \times \mathscr{B}(\mathbb{R}))$-measurable in $(t, \om, \tau)$ for any given $g_0 \in E$.

\textup{2)} For any $\om\in\gw$ and $t \geq \tau \in \mathbb{R}$, $(u(t, \omega, \tau, u_{0}), v(t, \omega, \tau, v_{0}))$ is weakly continuous with respect to $g_0 = (u_0, v_0)$ in $E$ in the sense that $$
    (u(t, \omega, \tau, u_{0,m}), v(t, \omega, \tau, v_{0,m})) \rightharpoonup (u(t, \omega, \tau, u_{0}), v(t, \omega, \tau, v_{0}))
$$
weakly in $E$, provided that $g_{0,m} = (u_{0,m}, v_{0,m}) \rightharpoonup g_0 = (u_0, v_0)$ weakly in $E$.
\end{lemma}

\begin{proof}
     The local existence and uniqueness of a weak solution to the $\omega$-parametrized PDE problem (\ref{310}) in the phase space $E=(H^1(\mathbb{R}^n) \cap L^p(\mathbb{R}^n)) \times L^2(\mathbb{R}^n)$ and its weakly continuous dependence on the initial data can be established by the Galerkin approximation method as in \cite[Chapter XV]{ChepyVishik02} and the result in Lemma 3.1 of \cite{Wang11}. Also see \cite{te, Wang, WangZhouGu}. The proof of the global existence of weak solutions will be included in the proof of Lemma \ref{L41} below. Here it is omitted.
\end{proof}
The solution mapping can be used to define a random dynamical system. 

\begin{definition} \label{sflow}
    A family of mappings $S(t, \tau; \omega): X \to X$ on a Banach space $X$ for $t \geq \tau \in \mathbb{R}$ and $\omega\in\Omega$ is called a stochastic semiflow, if it satisfies the following properties:

    (i) $S(t, s; \omega) S(s, \tau; \omega) = S(t, \tau; \omega)$, for all $\tau \leq s \leq t$ and $\omega\in\Omega$;

    (ii) $S(t, \tau; \omega) = S(t-\tau, 0; \theta_{\tau} \omega)$, for all $\tau \leq t$ and $\omega\in\Omega$;

    (iii) The mapping $S(t, \tau; \omega)x$ is measurable in $(t, \tau, \omega)$ and continuous in $x \in X$.
\end{definition}

Here for this problem \eqref{22} and the converted version \eqref{310} we define
\begin{equation} \label{ss}
    \begin{split}
    S(t, \tau; \omega) (u_0, v_0) &= (u, v)(t, \omega, \tau, (u_0, v_0)) \\[2pt]
    &= (u (t, \omega, \tau, u_0), v(t, \omega, \tau, u_1 + \delta u_0 - \varepsilon z(\theta_\tau \omega))),
    \end{split}
\end{equation}
where $(u, v) (t, \omega, \tau, (u_0, v_0))$ is the weak solution of the initial value problem \eqref{310}, shown above. This mapping $S(t, \tau; \omega): E \to E$ is a stochastic semiflow. Then define a mapping $\Phi: \mathbb{R}^+ \times \Omega \times E \to E$ by
\begin{equation} \label{CC}
    \Phi (t-\tau, \theta_\tau \omega, (u_0, v_0)) = S(t, \tau; \, \omega) (u_0, v_0),
\end{equation}
which is equivalent to
\begin{equation} \label{ccl}
    \begin{split}
    \Phi (t, \omega, (u_0, v_0)) &= S(t, 0; \omega) (u_0, v_0) \\[2pt]
    &= (u (t, \omega, 0, u_0), v(t, \omega, 0, u_1 + \delta u_0 - \varepsilon z(\omega))).
    \end{split}
\end{equation}
\begin{lemma} \label{Lcocycle}
    The mapping $\Phi: \mathbb{R}^+ \times \Omega \times E \to E$ defined by \eqref{CC} is a random dynamical system \textup{(}or called a cocycle\textup{)} on $E$ over the canonical parametric dynamical system $(\Omega, \mathcal{F}, P, \{\theta_t\})$. Moreover,
\begin{equation} \label{pbj}
    \begin{split}
    \Phi (t, \theta_{-t} \omega, (u_0, v_0)) &= (u(0, \omega, -t, u_{0}), v(0, \omega, -t, v_{0})) \\[1pt]
        &= (u(0, \omega, -t, u_{0}), \xi(0, \omega, -t, \xi_{0})-\varepsilon z(\omega)), \quad t \geq 0,
    \end{split}
\end{equation}
will be called a pullback quasi-trajectory.
\end{lemma}

\section{\large{\bf{Uniform Estimates of Pullback Quasi-Trajectories}}}

In this section, we shall derive uniform estimates on the solutions of the random wave equation (\ref{310}) defined on $\mathbb{R}^n$ in a long run as $t\rightarrow \infty$.
These \emph{a priori} estimates pave the way to proving the existence of  pullback absorbing sets and the pullback asymptotic compactness of the cocycle $\Phi$. In particular, we will show that the tails of the solutions for large spatial variables are uniformly small when time is sufficiently large.

Define a new  norm of $E$ by
\begin{equation} \label{Enorm}
\|(u, v)\|_E= \left(\|v\|^2+(\alpha+\delta^2- \beta \delta)\|u\|^2+\|\nabla u\|^2\right)^{\frac 12} + \|u\|_{L^p},
\end{equation}
in which and hereafter let $\delta$ be a fixed positive constant satisfying
\begin{eqnarray}\label{40}
\alpha+\delta^2- \beta \delta>0 \qquad \textup{and} \qquad \beta - 3\delta > 0.
\end{eqnarray}
It is easy to see that $\|\cdot\|_E$ in \eqref{Enorm} and the Sobolev norm $\|\cdot\|_{(H^1\cap L^p)\times L^2}$ in (\ref{311}) are equivalent.

\textbf{Assumption III}.
Let the intensity $\varepsilon>0$ of the additive noise satisfy
\begin{eqnarray}\label{41}
    0 < \varepsilon < \frac{\delta C_2 C_3 p}{C_1(p-1)},
\end{eqnarray}
where $p > 2$ and $C_i (i=1, 2, 3)$ are the positive constants in Assumption II 
and $\delta$ is the fixed constant in \eqref{40}.

\subsection{\textbf{Pullback Absorbing Set}}

The next lemma shows that there exists a random absorbing set in the universe
$\msd = \msd_E$ for the random dynamical system $\Phi$ associated with the problem \eqref{310}.

\begin{lemma} \label{L41}
Under the Assumptions \textup{I, II} and \textup{III}, there exists a $\msd$-pullback absorbing set $K = \{K(\omega)\}_{\omega\in \Omega}\in \msd$ for the cocycle $\Phi$ associated with the problem \eqref{310}. For any $B=\{B(\omega)\}_{ \omega\in \Omega}\in \msd$ and 
$\omega\in \Omega$, there exists a finite $T_B (\omega) >0$, such that
\begin{eqnarray*}
    \Phi (t, \theta_{-t}\omega, B(\theta_{-t}\omega))\subset K(\omega), \quad \textup{for all}\;\; t\geq T_B (\omega).
\end{eqnarray*}
\end{lemma}

\begin{proof}
 Taking the inner product of the second equation of (\ref{310}) with $v$ in $L^2(\mathbb{R}^n)$, we get
\begin{equation}\label{42}
    \begin{split}
    &\frac 12\frac{d}{dt}\|v\|^2- \delta\|v\|^2+(\alpha+\delta^2) \inpt{u, v} +\inpt{Au, v}+\inpt{f(x, u), v}  \\[3pt]
    =&\, -\inpt{\beta (v-\delta u+\varepsilon z(\theta_t\omega)), v}+2\delta\varepsilon\inpt{z(\theta_t\omega), v}+\inpt{g(x), v}.
    \end{split}
\end{equation}
By the first equation of (\ref{310}), we have
\begin{eqnarray}\label{43}
v=u_t-\varepsilon z(\theta_t\omega)+\delta u.
\end{eqnarray}
and
\begin{eqnarray}\label{44}
  -\inpt{\beta (v+\varepsilon z(\theta_t\omega)-\delta u), v} \leq -\beta \|v\|^2+ \beta \varepsilon \|z(\theta_t\omega)\| \|v\|+ \beta \delta \inpt{u, v}.
\end{eqnarray}
Substituting the above (\ref{43}) into the third and fourth terms on the left-hand side of (\ref{42}), we find that
\begin{eqnarray}\label{45}
\inpt{u, v} = \inpt{u, u_t+\delta u-\varepsilon z(\theta_t\omega)} \geq \frac 12\frac{d}{dt}\|u\|^2+\delta\|u\|^2-\varepsilon\|z(\theta_t\omega)\| \|u\|
\end{eqnarray}
and
\begin{equation}\label{46}
    \begin{split}
    \inpt{Au, v} & = \inpt{\nabla u, \nabla v} = \inpt{\nabla u, \nabla u_t+\delta \nabla u-\varepsilon \nabla z(\theta_t\omega)} \\[4pt]
    &\geq \frac 12\frac{d}{dt}\|\nabla u\|^2+\delta\|\nabla u\|^2-\varepsilon \|\nabla z(\theta_t\omega)\| \|\nabla u\|.
    \end{split}
\end{equation}
For the last term on the left-hand side of (\ref{42}), we have
\begin{equation}\label{47}
    \begin{split}
    &\inpt{f(x, u), v}=\inpt{f(x, u), u_t+\delta u - \varepsilon z(\theta_t\omega)}  \\[4pt]
    =&\, \frac{d}{dt}\int_{\mathbb{R}^n}F(x, u)dx + \delta\inpt{f(x, u), u}-\varepsilon \inpt{f(x, u), z(\theta_t\omega)}.
    \end{split}
\end{equation}
By the Assumption II, we get
\begin{equation}\label{48}
    \begin{split}
    &\delta\inpt{f(x, u), u}-\inpt{f(x, u), \varepsilon z(\theta_t\omega)}  \\[7pt]
    \geq &\, \delta C_2\int_{\mathbb{R}^n}F(x, u)dx + \delta\int_{\mathbb{R}^n}\phi_2(x)dx-\varepsilon C_1\int_{\mathbb{R}^n}|u|^{p-1}|z(\theta_t\omega)|dx \\[4pt]
    &\, -\varepsilon\int_{\mathbb{R}^n}|\phi_1(x)||z(\theta_t\omega)|dx  \\[4pt]
    \geq &\, \delta C_2\int_{\mathbb{R}^n}F(x, u)dx + \delta\|\phi_2\|_{L^1} -\frac{\varepsilon C_1(p-1)}{p}\|u\|^{p}_{L^p}-\frac{\varepsilon C_1}{p}\|z(\theta_t\omega)\|^{p}_{L^p} \\[4pt]
    &\, -\frac \varepsilon 2\|\phi_1\|^2-\frac \varepsilon 2\|z(\theta_t\omega)\|^2  \\[4pt]
    \geq &\, \delta C_2\int_{\mathbb{R}^n}F(x, u)dx + \delta\|\phi_2\|_{L^1}-\frac{\varepsilon C_1(p-1)}{C_3\, p}\int_{\mathbb{R}^n}F(x, u)\, dx -\frac{\varepsilon C_1(p-1)}{C_3\, p}\|\phi_3\|_{L^1} \\[4pt]
    &\, -\frac{\varepsilon C_1}{p}\|z(\theta_t\omega)\|^{p}_{L^p} -\frac \varepsilon2\|\phi_1\|^2-\frac \varepsilon2\|z(\theta_t\omega)\|^2  \\[4pt]
    \geq &\, \left(\delta C_2-\frac{\varepsilon C_1(p-1)}{C_3\, p}\right) \int_{\mathbb{R}^n}F(x, u)\,dx- \left(\frac{\varepsilon C_1}{p}\|z(\theta_t\omega)\|^{p}_{L^p} +\frac \varepsilon2\|z(\theta_t\omega)\|^2\right)-C_4,
    \end{split}
\end{equation}
where
$$
    C_4 = \frac{\varepsilon}{2} \|\phi_1\|^2 - \delta \|\phi_2 \|_{L^1} + \frac{\varepsilon C_1 (p -1)}{C_3\, p} \|\phi_3\|_{L^1}.
$$
For the last term on the right-hand side of (\ref{42}),
\begin{eqnarray}\label{49}
\inpt{g, v}\leq \|g\| \|v\|\leq \frac{\|g\|^2}{2(\beta-\delta)} + \frac{\beta-\delta}{2}\|v\|^2.
\end{eqnarray}
Substitute (\ref{44})-(\ref{49}) into (\ref{42}). Then we obtain

\begin{eqnarray}\label{410}
&&\frac 12\frac{d}{dt}\|v\|^2-\delta\|v\|^2+\frac12 \left(\alpha+\delta^2- \beta \delta \right) \frac{d}{dt}\|u\|^2+ \delta \left(\alpha+\delta^2- \beta \delta \right)\|u\|^2\nonumber\\[6pt]
&&- \,\varepsilon \left(\alpha+\delta^2-\beta \delta \right) \|z(\theta_t\omega)\|\|u\|+ \frac 12\frac{d}{dt}\|\nabla u\|^2+\delta \|\nabla u\|^2\nonumber \\[6pt]
&& -\,\varepsilon \|\nabla z(\theta_t\omega)\|\|\nabla u\|+\frac{d}{dt} \int_{\mathbb{R}^n} F(x, u)\, dx + \left(\delta C_2-\frac{\varepsilon
 C_1(p-1)}{C_3\, p} \right) \int_{\mathbb{R}^n}F(x, u)\, dx \nonumber\\[5pt]
&& -\left(\frac{\varepsilon C_1}{p}\|z(\theta_t\omega)\|^{p}_{L^p}+\frac\varepsilon2 \|z(\theta_t\omega)\|^2 \right)-C_4 +\beta \|v\|^2- \beta \varepsilon\|z(\theta_t\omega)\|\|v\|
\nonumber\\[5pt]
&& \leq \frac{\|g\|^2}{2(\beta -\delta)}+ \frac{\beta - \delta}{2}\|v\|^2+2\varepsilon \delta\|z(\theta_t\omega)\|\|v\|.
\end{eqnarray}
Grouping some terms together in \eqref{410},  we obtain
\begin{equation}\label{411}
    \begin{split}
    &\frac 12\frac{d}{dt}\left[\|v\|^2+ \left(\alpha+\delta^2-\beta\delta \right)\|u\|^2+\|\nabla u\|^2+2\int_{\mathbb{R}^n}F(x, u)\, dx\right]   \\[5pt]
    & +\, \frac{\delta}{2} \left[\|v\|^2+ \left(\alpha+\delta^2-\beta \delta \right)\|u\|^2+\|\nabla u\|^2 \right]  \\[5pt]
    & + \left(\delta C_2-\frac{\varepsilon C_1(p-1)}{C_3\, p}\right) \int_{\mathbb{R}^n} F(x, u)\, dx   \\[5pt]
    \leq &\, \frac{3\delta-\beta}{2}\|v\|^2+\frac{\|g\|^2}{2(\beta-\delta)}+ \frac{\varepsilon^2(\alpha+\delta^2-\beta\delta)}{2\delta}\, \|z(\theta_t\omega)\|^2 +\frac{\varepsilon^2}{2\delta}\, \|\nabla z(\theta_t\omega)\|^2   \\[5pt]
    & +\, \frac{\varepsilon^2(2\delta+\beta)^2}{2\delta}\, \|z(\theta_t\omega)\|^2 + \frac{\varepsilon C_1}{p}\, \|z(\theta_t\omega)\|^{p}_{L^p}+\frac \varepsilon2\|z(\theta_t\omega)\|^2+C_4  \\[4pt]
    \leq& \frac{\|g\|^2}{2(\beta -\delta)} + C_0 \left(\|z(\theta_t\omega)\|^2+\|\nabla z(\theta_t\omega)\|^2 +\|z(\theta_t\omega)\|^{p}_{L^p} \right) + C_4 ,
    \end{split}
\end{equation}
where $C_0 > 0$ is a constant and by \eqref{40} the term $(3\delta - \beta)\|v\|^2/2 \leq 0$. Since $z(\theta_t\omega)=\sum^m_{j=1}h_jz_j(\theta_t\omega_j)$ and $h_j\in H^2(\mathbb{R}^n)\cap W^{2, p}(\mathbb{R}^n)$,  by \eqref{38} and \eqref{39} we see that there is a constant $C_5 > 0$ such that for $P$-a.e. $\omega\in\Omega,$
\begin{equation} \label{413}
    \begin{split}
    &\Gamma_1(\theta_t \, \omega) := C_0 \left(\|z(\theta_t \omega)\|^2+\|\nabla z(\theta_t \omega)\|^2 +\|z(\theta_t \omega)\|^{p}_{L^p} \right) \\[3pt]
    \leq &\, c\, \sum^m_{j=1}\left(\|z_j(\theta_t\omega_j)\|^2 + \|z_j(\theta_t\omega_j)\|_{L^p}^{p}\right) \leq C_5\, e^{\frac 12 \sigma |t|} r_0 (\omega), \quad \textup{for all} \;\; t \in \mathbb{R}.
    \end{split}
\end{equation}
It follows from (\ref{411})-(\ref{413}) that

\begin{equation*}
    \begin{split}
    &\, \frac{d}{dt}\left[\|v\|^2+ \left(\alpha+\delta^2-\beta\delta\right)\|u\|^2+\|\nabla u\|^2+2\int_{\mathbb{R}^n} F(x, u)\, dx\right]   \\
    &\, +\, \delta \left[\|v\|^2+ \left(\alpha+\delta^2-\beta\delta\right)\|u\|^2+\|\nabla u\|^2\right] +2\left(\delta C_2-\frac{\varepsilon C_1(p-1)}{C_3\, p}\right) \int_{\mathbb{R}^n} F(x, u)\, dx  \\
    &\leq \frac{\|g\|^2}{\beta-\delta}+2\Gamma_1(\theta_t\omega)+2C_4, \quad t \in \mathbb{R}, \;\, \om \in \gw.
    \end{split}
\end{equation*}
It leads to the differential inequality
\begin{equation}\label{414}
    \begin{split}
    &\, \frac{d}{dt}\left[\|v\|^2+ \left(\alpha+\delta^2-\beta\delta\right)\|u\|^2+\|\nabla u\|^2+2\int_{\mathbb{R}^n} (F(x, u) + \phi_3 (x))\, dx\right]   \\[5pt]
    &\, +\, \delta \left[\|v\|^2+ \left(\alpha+\delta^2-\beta\delta\right)\|u\|^2+\|\nabla u\|^2 \right]  \\[5pt]
    &\, + 2 \left(\delta C_2 - \frac{\varepsilon C_1 (p-1)}{C_3 p} \right)\int_{\mathbb{R}^n} (F(x, u) + \phi_3 (x))\, dx   \\[5pt]
    \leq &\, \frac{\|g\|^2}{\beta-\delta}+2\Gamma_1(\theta_t\omega) + 2\left(\delta C_2-\frac{\varepsilon C_1(p-1)}{C_3\, p}\right) \|\phi_3 \|_{L^1} +2C_4.
    \end{split}
\end{equation}
By (\ref{41}) in the Assumption III, let $\sigma$ be a fixed positive constant:
\begin{equation} \label{sigma}
    \sigma=\min \left\{\delta, \, \delta C_2-\frac{\varepsilon C_1(p-1)}{C_3\, p} \right\} > 0.
\end{equation}
We have $\int_{\mathbb{R}^n} (F(x, u) + \phi_3 (x))\, dx \geq 0$ due to \eqref{35} in the Assumption II. It follows from \eqref{414} and \eqref{sigma} that
\begin{equation} \label{sgre}
    \begin{split}
    &\, \frac{d}{dt}\left[\|v\|^2+ \left(\alpha+\delta^2-\beta\delta\right)\|u\|^2+\|\nabla u\|^2+2\int_{\mathbb{R}^n} (F(x, u) + \phi_3 (x))\, dx\right]   \\
    &\, + \sigma \left[\|v\|^2+ \left(\alpha+\delta^2-\beta\delta\right)\|u\|^2+\|\nabla u\|^2 +2 \int_{\mathbb{R}^n} (F(x, u) + \phi_3 (x))\, dx\right]  \\
    \leq &\, \frac{\|g\|^2}{\beta-\delta}+2\Gamma_1(\theta_t\omega) + C_6, \quad t \in \mathbb{R}, \; \, \om \in \gw,
    \end{split}
\end{equation}
where $C_6 = 2\left(\delta C_2-\frac{\varepsilon C_1(p-1)}{C_3\, p}\right) \|\phi_3 \|_{L^1} +2C_4$.
Thus we can apply the Gronwall lemma to the above inequality to obtain that, for any $\om \in \gw$ and $t \geq \tau$, the solution of \eqref{310} satisfies
\begin{equation}\label{415}
    \begin{split}
    &\, \|v(t, \omega, \tau, v_0)\|^2+ \left(\alpha+\delta^2-\beta \delta\right)\|u(t, \omega, \tau, u_0)\|^2   \\[4pt]
    &\, + \|\nabla u(t, \omega, \tau, u_0 )\|^2+2\int_{\mathbb{R}^n} (F(x, u( t, \omega, \tau, u_0 )) + \phi_3 (x))\, dx  \\
    \leq &\, e^{-\sigma (t - \tau)} \left[\|v_0\|^2+ \left(\alpha+\delta^2-\beta \delta\right)\|u_0\|^2 +\|\nabla u_0\|^2+2\int_{\mathbb{R}^n}F(x, u_0)\, dx\right]  \\
    &\, + 2e^{-\sigma (t - \tau)} \|\phi_3\|_{L^1} + 2\int^t_\tau e^{\sigma(s - t)} \Gamma_1(\theta_{s}\omega)\, ds +\frac{1}{\sigma} \left(C_6 +\frac{\|g\|^2}{\beta -\delta}\right).
    \end{split}
\end{equation}
Replace the time interval $[\tau, t]$ by $[-t, 0]$ and consider any given tempered set $B \in \msd$. For $(u_{0}(\ctw), v_0(\ctw))\in B(\theta_{-t}\omega)$, we obtain
\begin{equation}\label{416}
    \begin{split}
    &\|v(0, \, \omega, -t, v_0(\theta_{-t}\omega))\|^2+ \left(\alpha+\delta^2-\beta\delta\right)\|u(0, \, \omega, -t, u_0(\ctw) )\|^2 \\[6pt]
    &+ \, \|\nabla u(0, \, \omega, -t, u_0(\ctw) )\|^2\\[6pt]
    &+2\int_{\mathbb{R}^n} (F(x, u(0, \, \omega, -t, u_0 (\ctw))) + \phi_3 (x))\, dx \\[6pt]
    \leq &\, e^{-\sigma t} \left[\|v_0(\theta_{-t}\omega)\|^2+ \left(\alpha+\delta^2-\beta\delta\right) \|u_0(\ctw)\|^2+\|\nabla u_0(\ctw)\|^2 \right] \\[6pt]
    &+ \,2e^{-\sigma t} \left[\int_{\mathbb{R}^n} F(x, u_0(\ctw))\, dx + \|\phi_3\|_{L^1} \right] \\
    &+ 2\int_{-t}^0 e^{\sigma s} \Gamma_1(\theta_{s}\omega)\, ds + \frac{1}{\sigma} \left(C_6 +\frac{\|g\|^2}{\beta -\delta}\right)  \\
    \leq &\, e^{-\sigma t} \left[\|v_0(\theta_{-t}\omega)\|^2+ \left(\alpha+\delta^2-\beta \delta\right) \|u_0(\theta_{-t}\omega)\|^2+\|\nabla u_0(\theta_{-t}\omega)\|^2+2\int_{\mathbb{R}^n}F(x, u_0(\theta_{-t}\omega))\, dx\right]\\
    &+ \, 2e^{-\sigma t}\|\phi_3\|_{L^1} + 2C_5\int^0_{-t} e^{\sigma s + \frac 12 \sigma |s|} r_0(\omega)\, ds+ \frac{1}{\sigma} \left(C_6 +\frac{\|g\|^2}{\beta -\delta}\right) \\
    = &\, e^{-\sigma t} \left[\|v_0(\theta_{-t}\omega)\|^2+ \left(\alpha+\delta^2-\beta\delta\right)\|u_0(\theta_{-t}\omega)\|^2+\|\nabla u_0(\theta_{-t}\omega)\|^2+2\int_{\mathbb{R}^n}F(x, u_0(\theta_{-t}\omega))\, dx\right]  \\
    &+ 2e^{-\sigma t}\|\phi_3\|_{L^1} + \frac{1}{\sigma} \left(4C_5\, r_0 (\omega) + C_6 +\frac{\|g\|^2}{\beta -\delta}\right), \quad  t \geq 0, \;\; \om \in \gw.
    \end{split}
\end{equation}
Note that \eqref{33} and \eqref{34} imply that there is a constant $c = c (C_1, C_2, \phi_1, \phi_2) > 0$ such that
\begin{eqnarray}\label{417}
    \int_{\mathbb{R}^n}F(x, u_0(\ctw))\, dx\leq c \left(1 +\|u_0(\ctw)\|^2 +\|u_0(\ctw)\|^{p}_{L^p}\right).
\end{eqnarray}
For any set $B \in \msd$, which is tempered with respect to the norm of $E$,  since $(u_{0}(\ctw), v_0(\theta_{-t}\omega))\in B(\theta_{-t}\omega)$, there exists a constant $C > 0$ and a finite
$T_B(\omega)>0$ such that for all $t \geq T_B (\omega)$ one has
\begin{equation} \label{418}
    \begin{split}
    &\, e^{-\sigma t} \left[\|v_0(\theta_{-t}\omega)\|^2+\left(\alpha+\delta^2-\beta \delta\right)\|u_0(\ctw)\|^2+\|\nabla u_0(\ctw)\|^2\right] \\[3pt]
    & + 2e^{-\sigma t}\left[ \int_{\mathbb{R}^n}F(x, u_0(\ctw))\, dx + \|\phi_3\|_{L^1}\right] \\[2pt]
    \leq &\, C e^{-\sigma t} \left(1 + \|v_0(\theta_{-t}\omega)\|^2+ \|u_0(\ctw)\|^2_{H^1} +\|u_0(\ctw)\|^p_{L^p}\right) \leq 1.
    \end{split}
\end{equation}
Substitute \eqref{418} into the right-hand side of the last equality of \eqref{416} and note that \eqref{35} implies
$$
    2\int_{\mathbb{R}^n} (F(x, u(0, \, \omega, -t, u_0(\ctw)) + \phi_3 (x))\, dx \geq 2C_3 \|u(0, \, \omega, -t, u_0(\ctw))\|_{L^p}^p.
$$
Then it results in
\begin{equation}\label{420}
    \begin{split}
    &\|v(0, \, \omega, -t, v_0(\theta_{-t}\omega))\|^2+ \left(\alpha+\delta^2-\beta\delta\right)\|u(0, \, \omega, -t, u_0(\ctw))\|^2 \\[7pt]
    &+ \, \|\nabla u(0, \, \omega, -t, u_0(\ctw))\|^2 + 2C_3\|u(0, \, \omega, -t, u_0(\ctw))\|_{L^p}^p \\[3pt]
    \leq &\, 1 + \frac{1}{\sigma} \left(4C_5\, r_0 (\omega) + C_6 +\frac{\|g\|^2}{\beta -\delta}\right).
    \end{split}
\end{equation}
By \eqref{pbj} and \eqref{420}, we conclude that $\Phi (t, \theta_{-t} \omega, B(\theta_{-t} \omega)) \subset K(\omega) = B_E (0, R(\omega))$
 for $t \geq T_B (\omega)$, where the radius of the ball $B_E (0, R(\omega))$ in the space $E$ is
\begin{equation} \label{Kom}
    \begin{split}
    R(\omega) &= \left(\frac{1}{\min \{1, (\alpha + \delta^2 -\beta \delta)\}} \left[1 + \frac{1}{\sigma} \left(4C_5\, r_0 (\omega) + C_6 +\frac{\|g\|^2}{\beta -\delta}\right)\right]\right)^{\frac{1}{2}}  \\
    &+ \left(\frac{1}{2C_3}\left[1 + \frac{1}{\sigma} \left(4C_5\, r_0 (\omega) + C_6 +\frac{\|g\|^2}{\beta -\delta}\right)\right]\right)^{\frac{1}{p}}.
    \end{split}
\end{equation}
Since $r_0 (\omega)$ is a tempered random variable, $K = \{K(\omega)\}_{\omega\in \Omega}\in \msd$. Therefore, $K = \{K(\omega)\}_{\omega\in \Omega}$ is a $\msd$-pullback absorbing set for the cocycle $\Phi$.
\end{proof}

\subsection{\textbf{Tail Estimates}}

Next we conduct uniform estimates on the tail parts of the solutions for large spatial and time variables. These estimates play key roles in proving the pullback asymptotic compactness of the random dynamical systems $\Phi$ associated with the random wave equation (\ref{310}) on the unbounded domain $\mathbb{R}^n.$

\begin{lemma} \label{L42}
 Under the Assumptions \textup{I, II} and \textup{III}, for every $B=\{B(\omega)\}_{\omega\in \Omega} \in \msd, \, 0 < \eta \leq 1$ and $P$-a.e. $\omega\in \Omega$,
there exists $T=T(\omega,B, \eta)>0$ and $V = V (\omega, \eta)\geq 1$ such that the cocycle $\Phi$ associated with the problem \eqref{310} satisfies
\begin{eqnarray}\label{431}
    \| \Phi (t, \theta_{-t}\omega, B(\theta_{-t}\omega))\|_{E(\mathbb{R}^n \backslash B_r)} = \max_{g_0 \in B(\theta_{-t} \omega)} \|\Phi (t, \theta_{-t} \omega, g_0)\chi_{B_r^c}\|_E < \eta,
\end{eqnarray}
for all $t\geq T$ and every $r\geq V$, where $\chi_{B_r^c } (x)$ is the characteristic function of the set $\{x\in \mathbb{R}^n : |x| > r\}$.
\end{lemma}

\begin{proof}
We choose a smooth nondecreasing function $\rho$ such that $0\leq \rho (s) \leq 1$ for all $s \in [0, \infty)$ and
\begin{eqnarray}\label{432}
\rho(s)=\left\{
\begin{array}{l}
0, \qquad if \;\; 0 \leq s <1,\\ \\
1, \qquad if \;\; s >2,
\end{array} \right.
\end{eqnarray}
with $0 \leq \rho'(s) \leq 2$ for $s\geq 0$. Taking the inner product of the second equation of (\ref{310}) with $\rho(|x|^2/r^2)v$ in $L^2(\mathbb{R}^n)$, we get
\begin{equation}\label{433}
    \begin{split}
    &\frac 12\frac{d}{dt}\int_{\mathbb{R}^n}\rho\left(\frac{|x|^2}{r^2}\right)|v|^2\, dx -\delta \int_{\mathbb{R}^n}\rho\left(\frac{|x|^2}{r^2}\right)|v|^2 \, dx
+(\alpha+\delta^2)\int_{\mathbb{R}^n}\rho\left(\frac{|x|^2}{r^2}\right)uv \, dx  \\[1pt]
    &+\, \int_{\mathbb{R}^n}(Au)\rho\left(\frac{|x|^2}{r^2}\right)v \, dx +\int_{\mathbb{R}^n}\rho\left(\frac{|x|^2}{r^2}\right)f(x, u) v \, dx \\[1pt]
    =& \int_{\mathbb{R}^n}\rho\left(\frac{|x|^2}{r^2}\right)(g+2\delta \varepsilon z(\theta_{t}\omega))\, v\, dx -\int_{\mathbb{R}^n}\rho\left(\frac{|x|^2}{r^2}\right) \beta (v+ \varepsilon z(\theta_{t}\omega)-\delta u) v\, dx.
    \end{split}
\end{equation}
Substitute
\begin{equation*} 
    \begin{split}
    &- \int_{\mathbb{R}^n}\rho\left(\frac{|x|^2}{r^2}\right) \beta (v+ \varepsilon z(\theta_{t}\omega)-\delta u)v\, dx \\[3pt]
    \leq - \beta \int_{\mathbb{R}^n}\rho\left(\frac{|x|^2}{r^2}\right)|v|^2&\, dx
    + \beta\delta\int_{\mathbb{R}^n}\rho\left(\frac{|x|^2}{r^2}\right)uv\, dx + \varepsilon \beta\int_{\mathbb{R}^n}\rho\left(\frac{|x|^2}{r^2}\right)| z(\theta_t\omega)||v|\, dx.
    \end{split}
\end{equation*}
into (\ref{433}) to obtain
\begin{equation}\label{435}
    \begin{split}
    &\frac 12\frac{d}{dt}\int_{\mathbb{R}^n}\rho\left(\frac{|x|^2}{r^2}\right)|v|^2\, dx  + (\alpha+\delta^2-\beta\delta)\int_{\mathbb{R}^n}\rho\left(\frac{|x|^2}{r^2} \right)uv\, dx  \\[3pt]
       + &\, (\beta -\delta) \int_{\mathbb{R}^n}\rho\left(\frac{|x|^2}{r^2}\right)|v|^2\, dx + \int_{\mathbb{R}^n}(Au)\rho\left(\frac{|x|^2}{r^2}\right)v\, dx+\int_{\mathbb{R}^n}\rho\left(\frac{|x|^2}{r^2}\right)f(x, u)\,v\, dx  \\[3pt]
    \leq &\, \int_{\mathbb{R}^n}\rho\left(\frac{|x|^2}{r^2}\right)\left(\frac{\delta}{2} |v|^2 + \frac{\varepsilon^2 \beta^2}{2\delta}|z(\theta_t\omega)|^2\right) dx +\int_{\mathbb{R}^n}\rho\left(\frac{|x|^2}{r^2}\right)(g+2\delta \varepsilon z(\theta_t\omega))\,v \, dx.
    \end{split}
\end{equation}
\noindent For the second term on the left-hand side of (\ref{435}), by the first equation of \eqref{310} we have
\begin{equation}\label{2T}
    \begin{split}
    &(\alpha+\delta^2-\beta\delta)\int_{\mathbb{R}^n}\rho\left(\frac{|x|^2} {r^2}\right)uv\, dx \\
    =&\, (\alpha+\delta^2-\beta\delta)\int_{\mathbb{R}^n}\rho\left(\frac{|x|^2}{r^2} \right)u(u_t+\delta u-\varepsilon z(\theta_{t}\omega))\, dx \\
    \geq &\, (\alpha+\delta^2-\beta\delta) \left(\frac 12\frac{d}{dt}\int_{\mathbb{R}^n}\rho\left(\frac{|x|^2}{r^2}\right)|u|^2\, dx +\delta\int_{\mathbb{R}^n}\rho\left(\frac{|x|^2}{r^2}\right)|u|^2\, dx\right) \\
    &\, - (\alpha+\delta^2-\beta\delta)\int_{\mathbb{R}^n}\rho\left(\frac{|x|^2}{r^2}\right) \left(\frac{\delta}{2}|u|^2 + \frac{\varepsilon^2}{2\delta} |z(\ztw)|^2\right) dx.
    \end{split}
\end{equation}
For the fourth term on the left-hand side of (\ref{435}), we have
\begin{align*}
    &\int_{\mathbb{R}^n}(Au)\, \rho\left(\frac{|x|^2}{r^2}\right)v \, dx = \int_{\mathbb{R}^n}(Au)\, \rho\left(\frac{|x|^2}{r^2}\right)(u_t+\delta u- \varepsilon z(\theta_{t}\omega)) \, dx \\[3pt]
    =&\, \int_{\mathbb{R}^n}(\nabla u)\nabla\left(\rho\left(\frac{|x|^2}{r^2}\right)(u_t+\delta u-\varepsilon z(\theta_{t}\omega))\right) dx \\[3pt]
    =&\, \int_{\mathbb{R}^n}(\nabla u)\left(\frac{2x}{r^2}\rho'\left(\frac{|x|^2}{r^2}\right)(u_t+\delta u-\varepsilon z(\theta_{t}\omega)) + \rho\left(\frac{|x|^2}{r^2}\right)\nabla(u_t+\delta u- \varepsilon z(\theta_{t}\omega))\right) dx  \\[3pt]
    =&\, \int_{\mathbb{R}^n}(\nabla u)\frac{2x}{r^2}\rho'\left(\frac{|x|^2}{r^2}\right)v\, dx+\int_{\mathbb{R}^n}(\nabla u)\rho\left(\frac{|x|^2}{r^2}\right)\nabla(u_t+\delta u-\varepsilon z(\theta_{t}\omega))\, dx  \\[3pt]
    =&\, \int_{\mathbb{R}^n}(\nabla u)\frac{2x}{r^2}\rho'\left(\frac{|x|^2}{r^2}\right)v\, dx+\frac 12\frac{d}{dt}\int_{\mathbb{R}^n}\rho\left(\frac{|x|^2}{r^2}\right)|\nabla u|^2\, dx
+\delta\int_{\mathbb{R}^n}\rho\left(\frac{|x|^2}{r^2}\right)|\nabla u|^2 \, dx   \\[3pt]
    &\, -\varepsilon \int_{\mathbb{R}^n}\rho\left(\frac{|x|^2}{r^2}\right)(\nabla u) (\nabla z(\theta_{t}\omega)) \, dx.
\end{align*}
Since $0 \leq \rho' (s) \leq 2$, it follows that
\begin{equation} \label{436}
    \begin{split}
    &\int_{\mathbb{R}^n}(Au)\, \rho\left(\frac{|x|^2}{r^2}\right)v\, dx \geq  -\, \int_{r\leq |x|\leq \sqrt{2}r}\frac{4|x|}{r^2}|(\nabla u)v|dx+\frac 12\frac{d}{dt}\int_{\mathbb{R}^n}\rho\left(\frac{|x|^2}{r^2}\right)|\nabla u|^2\, dx \\
&\, +\delta\int_{\mathbb{R}^n}\rho\left(\frac{|x|^2}{r^2}\right)|\nabla u|^2 \, dx - \varepsilon\int_{\mathbb{R}^n}\rho\left(\frac{|x|^2}{r^2}\right)(\nabla u) (\nabla z(\theta_{t}\omega))\, dx  \\
    &\geq -\, \frac{2\sqrt{2}}{r}  \int_{r \leq |x| \leq \sqrt{2}r}  (|\nabla u|^2+ |v |^2)\, dx  +\frac 12\frac{d}{dt}\int_{\mathbb{R}^n}\rho\left(\frac{|x|^2}{r^2}\right)|\nabla u|^2\, dx \\
&\, +\frac{\delta}{2}\int_{\mathbb{R}^n}\rho\left(\frac{|x|^2}{r^2}\right)|\nabla u|^2 \, dx -\frac{\varepsilon^2}{2\delta}\int_{\mathbb{R}^n}\rho\left(\frac{|x|^2}{r^2}\right) |\nabla z(\theta_{t}\omega))|^2\, dx.
    \end{split}
\end{equation}
For the fifth term on the left-hand side of (\ref{435}), by \eqref{33}-\eqref{35}, we have
\begin{equation}\label{437}
    \begin{split}
    &\, \int_{\mathbb{R}^n}\rho\left(\frac{|x|^2}{r^2}\right)f(x, u)v \, dx =\int_{\mathbb{R}^n}\rho\left(\frac{|x|^2}{r^2}\right)f(x, u)(u_t+\delta u- \varepsilon z(\theta_t \omega) ) \, dx \\[2pt]
    \geq &\, \frac{d}{dt}\int_{\mathbb{R}^n}\rho\left(\frac{|x|^2}{r^2}\right)F(x, u)\, dx +\delta \int_{\mathbb{R}^n}\rho\left(\frac{|x|^2}{r^2}\right) (C_2 F(x, u) + \phi_2(x))\, dx  \\[2pt]
    &\, -\varepsilon C_1\int_{\mathbb{R}^n}\rho\left(\frac{|x|^2}{r^2}\right)|u|^{p -1}|z(\theta_t\omega)|dx
-\varepsilon\int_{\mathbb{R}^n}\rho\left(\frac{|x|^2}{r^2}\right)|\phi_1(x)| |z(\theta_t\omega)|dx  \\[2pt]
    \geq &\, \frac{d}{dt}\int_{\mathbb{R}^n}\rho\left(\frac{|x|^2}{r^2}\right)F(x, u)\, dx +\delta \int_{\mathbb{R}^n}\rho\left(\frac{|x|^2}{r^2}\right) (C_2 F(x, u) + \phi_2(x))\, dx \\[2pt]
    &\, -\varepsilon C_1\frac{p-1}{p}\int_{\mathbb{R}^n}\rho\left(\frac{|x|^2}{r^2}\right)|u|^{p}\, dx-\frac{\varepsilon C_1}{p}\int_{\mathbb{R}^n}\rho\left(\frac{|x|^2}{r^2}\right)|z(\theta_t\omega)|^{p}\, dx  \\[2pt]
    &\, -\frac{\varepsilon}{2}\int_{\mathbb{R}^n}\rho\left(\frac{|x|^2}{r^2}\right) |z(\theta_t\omega)|^2dx -\frac{\varepsilon}{2}\int_{\mathbb{R}^n}\rho\left(\frac{|x|^2}{r^2}\right)|\phi_1|^2dx \\[2pt]
    \geq &\, \frac{d}{dt}\int_{\mathbb{R}^n}\rho\left(\frac{|x|^2}{r^2}\right)F(x, u)\, dx +\left(\delta C_2-\frac{\varepsilon C_1(p-1)}{C_3\, p}\right)\int_{\mathbb{R}^n}\rho
\left(\frac{|x|^2}{r^2}\right)F(x, u)\, dx  \\[2pt]
    &\,  -\, C_7\int_{\mathbb{R}^n}\rho\left(\frac{|x|^2}{r^2}\right)(|z(\theta_t\omega)|^{p} +|z(\theta_t\omega)|^2)dx  \\[2pt]
    &\, - \int_{\mathbb{R}^n}\rho\left(\frac{|x|^2}{r^2}\right)\left(\frac{\varepsilon C_1(p-1)}{C_3\, p}|\phi_3|-\delta|\phi_2|+\frac \varepsilon2 |\phi_1|^2\right) dx,
    \end{split}
\end{equation}
where $C_7 = C_7 (\varepsilon) > 0$ is a constant and we used the H\"{o}lder inequality in the second inequality as well as
$$
    -\varepsilon C_1\frac{p-1}{p}\int_{\mathbb{R}^n}\rho\left(\frac{|x|^2}{r^2}\right)|u|^{p}\, dx \geq -\frac{\varepsilon C_1(p-1)}{C_3p}\int_{\mathbb{R}^n}\rho\left(\frac{|x|^2}{r^2}\right) (F(x, u) + \phi_3 (x))\, dx
$$
in the third inequality. For the last term on the right-hand side of (\ref{435}), we have
\begin{equation} \label{rT}
    \begin{split}
    &\int_{\mathbb{R}^n}\rho\left(\frac{|x|^2}{r^2}\right)(g+2\delta \varepsilon z(\theta_t\omega))\,v\, dx  \\[3pt]
    \leq &\, \frac{1}{2(\beta -\delta)}\int_{\mathbb{R}^n}\rho\left(\frac{|x|^2}{r^2}\right) (|g|+2\delta \varepsilon |z(\theta_t\omega)|)^2\, dx+
\frac{\beta -\delta}{2}\int_{\mathbb{R}^n}\rho\left(\frac{|x|^2}{r^2}\right)|v|^2\, dx.
    \end{split}
\end{equation}
Now substitute \eqref{2T}-\eqref{rT} into \eqref{435}, we obtain
\begin{equation}\label{439}
    \begin{split}
    &\, \frac 12\frac{d}{dt}\int_{\mathbb{R}^n}\rho\left(\frac{|x|^2}{r^2}\right) \left(|v|^2+(\alpha+\delta^2-\beta \delta)|u|^2+|\nabla u|^2+2F(x, u)\right)dx \\[6pt]
    &\, +\frac{\delta}{2}\int_{\mathbb{R}^n}\rho\left(\frac{|x|^2}{r^2}\right) |v|^2\,dx +\frac{\delta}{2}\int_{\mathbb{R}^n}\rho\left(\frac{|x|^2}{r^2}\right)
((\alpha+\delta^2-\beta \delta)|u|^2+|\nabla u|^2)\, dx \\[6pt]
    &\, + \left(\delta C_2-\frac{\varepsilon C_1(p-1)}{C_3\, p}\right) \int_{\mathbb{R}^n}\rho\left(\frac{|x|^2}{r^2}\right) F(x, u)\, dx\\[6pt]
    \leq &\,  \frac{\varepsilon^2}{2\delta} \int_{\mathbb{R}^n}\rho\left(\frac{|x|^2}{r^2}\right) (|\nabla z(\theta_t\omega)|^2+(\alpha+\delta^2-\beta \delta) |z(\theta_t\omega)|^2+\beta^2 |z(\theta_t\omega)|^2)\,dx \\[6pt]
    &\, +\, \frac{2\sqrt{2}}{r}  \int_{r \leq |x| \leq \sqrt{2}r}  (|\nabla u|^2+ |v |^2)\, dx +C_7\int_{\mathbb{R}^n}\rho\left(\frac{|x|^2}{r^2}\right) (|z(\theta_t\omega)|^{p} +|z(\theta_t\omega)|^2)\,dx \\[6pt]
    &\, +\frac{1}{\beta -\delta}\int_{\mathbb{R}^n}\rho\left(\frac{|x|^2}{r^2}\right) (|g|^2+4\delta^2\varepsilon^2|z(\theta_t\omega)|^2)\, dx \\[6pt]
    &\, + \int_{\mathbb{R}^n}\rho\left(\frac{|x|^2}{r^2}\right) \left(\frac{\varepsilon C_1(p-1)}{C_3\, p}|\phi_3|-\delta|\phi_2|+\frac \varepsilon 2|\phi_1|^2\right) dx \\[6pt]
    \leq &\, \frac{2\sqrt{2}}{r} \int_{r \leq |x| \leq \sqrt{2}r}  (|\nabla u|^2+ |v |^2)\, dx  \\[6pt]
    &\, + C_8\,\int_{\mathbb{R}^n}\rho\left(\frac{|x|^2}{r^2}\right) (|\nabla z(\theta_t\omega)|^2+|z(\theta_t\omega)|^2+|z(\theta_t\omega)|^p)\, dx \\[6pt]
    &\, + C_9\int_{\mathbb{R}^n}\rho\left(\frac{|x|^2}{r^2}\right)(|g|^2+ |\phi_1|^2+|\phi_2|+|\phi_3|)\, dx,
    \end{split}
\end{equation}
where $C_8 = C_8 (\varepsilon) > 0$ and $C_9 = C_9 (\varepsilon) > 0$ are constants. In grouping the coefficients of the terms $\intr \rho (|x|^2/r^2)|v|^2\, dx$ appearing on both sides of \eqref{439}, we used \eqref{40} to get $(\beta - \delta)/2 = (\beta - 3\delta)/2 + \delta \geq \delta$.

Since $g, \phi_1\in L^2(\mathbb{R}^n)$ and $\phi_2, \phi_3 \in L^1(\mathbb{R}^n)$, for any given $\eta> 0,$ there exists $K_0 =
K_0 (\eta)\geq 1$ such that for all $r\geq K_0$,
\begin{equation}\label{440}
    \begin{split}
    &C_9 \int_{\mathbb{R}^n}\rho\left(\frac{|x|^2}{r^2}\right) \left(|g|^2 + |\phi_1|^2+|\phi_2| +|\phi_3|\right) dx  + 2\sigma \intr \rxr |\phi_3 (x)|\, dx \\[4pt]
    &\leq C_9 \int_{|x|\geq r} \left(|g|^2 + |\phi_1|^2+|\phi_2|+|\phi_3|\right) dx + 2\sigma \intr |\phi_3 (x)|\, dx \leq \eta.
    \end{split}
\end{equation}
By \eqref{sigma} and (\ref{439})-(\ref{440}), there exists $K_1 = K_1 (\eta) \geq 1$
such that for all $r\geq K_1$,
\begin{equation}\label{441}
    \begin{split}
    &\, \frac{d}{dt}\int_{\mathbb{R}^n}\rho\left(\frac{|x|^2}{r^2}\right)\left(|v|^2 +(\alpha+\delta^2-\beta\delta)|u|^2+|\nabla u|^2+2(F(x, u) + \phi_3)\right) dx  \\
    &\, + \sigma\int_{\mathbb{R}^n}\rho\left(\frac{|x|^2}{r^2}\right) \left(|v|^2+(\alpha+\delta^2-\beta\delta)|u|^2+|\nabla u|^2+2(F(x, u) + \phi_3)\right) dx  \\
    &\leq \eta \left(1 + \int_{r \leq |x| \leq \sqrt{2}r}  (|\nabla u|^2+ |v |^2)\, dx\right) \\
    &\, + C_8 \int_{|x| \geq r} \left(|\nabla z(\theta_t\omega)|^2+|z(\theta_t\omega)|^2+|z(\theta_t\omega)|^p \right) dx.
    \end{split}
\end{equation}
Integrating the above inequality \eqref{441} over the time interval $(-t, 0)$, we see that for any $t > 0, \om \in \gw$ and $r\geq K_1$,
\begin{equation}\label{E0}
    \begin{split}
    &\int_{\mathbb{R}^n}\rho\left(\frac{|x|^2}{r^2}\right)(|v(0, \omega, -t, v_{0}(\theta_{-t}\omega))|^2 + \left(\alpha+\delta^2-\beta\delta\right) |u(0, \omega, -t, u_{0}(\theta_{-t}\omega))|^2)\, dx  \\[4pt]
    & +\int_{\mathbb{R}^n}\rho\left(\frac{|x|^2}{r^2}\right) |\nabla u(0, \omega, -t, u_{0}(\theta_{-t}\omega))|^2dx \\[4pt]
    & + 2 \intr \rxr (F(x, u(0, \omega, -t, u_{0}(\theta_{-t}\omega)) + \phi_3 (x))\,dx \\[4pt]
    \leq &\, e^{-\sigma t} \int_{\mathbb{R}^n} \rxr (|v_0(\theta_{-t}\omega)|^2 + (\alpha+\delta^2-\beta\delta) |u_0(\theta_{-t}\omega)|^2 +  |\nabla u_0(\theta_{-t}\omega)|^2 )\, dx \\[6pt]
    &\, + 2e^{-\sigma t}\intr \rxr (F(x, u_0(\theta_{-t}\omega)) + \phi_3 (x))\, dx  + \frac{\eta}{\sigma} \\[4pt]
    &\, + \eta \int^0_{-t} e^{\sigma s} \int_{r \leq |x| \leq \sqrt{2}r} (|\nabla u(s, \omega, -t, u_{0} (\theta_{-t}\omega))|^2+ |v(s, \omega, -t, v_{0}(\theta_{-t}\omega))|^2)\, dx\, ds \\[4pt]
    &+C_8 \int^0_{-t} e^{\sigma s}\int_{|x| \geq r} (|\nabla z(\theta_{s}\omega)|^2+|z(\theta_{s}\omega)|^2+|z(\theta_{s} \omega)|^p)\, dx\,ds.
    \end{split}
\end{equation}
Next we conduct estimate of the terms on the right-hand side of in
(\ref{E0}). For the first two terms we have

\begin{equation}\label{444}
    \begin{split}
    &\, e^{-\sigma t} \int_{\mathbb{R}^n}\rho\left(\frac{|x|^2}{r^2}\right)(|v_0(\theta_{-t}\omega))|^2 + (\alpha+\delta^2-\beta\delta) |u_0(\theta_{-t}\omega))|^2 +  |\nabla u_0(\theta_{-t}\omega)|^2 )\, dx \\
    &\, + 2e^{-\sigma t}\intr \rxr (F(x, u_0(\theta_{-t}\omega)) + \phi_3 (x))\, dx \\
    \leq &\, e^{-\sigma t} \int_{\mathbb{R}^n} \left(|v_0(\theta_{-t}\omega)|^2 + (\alpha+\delta^2-\beta\delta) |u_0(\theta_{-t}\omega)|^2 +|\nabla u_0(\theta_{-t}\omega)|^2\right) dx \\
    & + 2e^{-\sigma t} \int_{\mathbb{R}^n} \left[\frac{1}{C_2}\left(C_1 |u_0(\theta_{-t}\omega)|^p + |u_0(\theta_{-t}\omega)| |\phi_1 (x)| + |\phi_2 (x)| \right) + |\phi_3 (x)| \right] dx  \\[6pt]
    \leq &\, e^{-\sigma t} \left(\|v_0(\theta_{-t}\omega)\|^2 +(\alpha+\delta^2-\beta\delta)\| u_0(\theta_{-t}\omega)\|^2+ \|\nabla u_0(\theta_{-t}\omega)\|^2\right)  \\[6pt]
    &+\, 2e^{-\sigma t} \left[\frac{1}{C_2} \left(\|u_0(\theta_{-t}\omega)\|^2 + C_1 \|u_0(\theta_{-t}\omega)\|^p_{L^p} + \|\phi_1 \|^2 + \|\phi_2 \|_{L^1}\right) + \|\phi_3 \|_{L^1}\right]  \\[6pt]
    \leq &\, C_{10} \, e^{-\sigma t} \left(\|v_0(\theta_{-t}\omega)\|^2 + \| u_0(\theta_{-t}\omega)\|^2+ \|\nabla u_0(\theta_{-t}\omega)\|^2 +\|u_0(\theta_{-t}\omega)\|_{L^p}^p \right)  \\[8pt]
    & +\, C_{10}\, e^{-\sigma t} \left(\|\phi_1 \|^2 + \|\phi_2 \|_{L^1} + \|\phi_3 \|_{L^1}\right) < \eta, \quad \textup{for all} \;\;  t \geq T_1.
    \end{split}
\end{equation}
where $C_{10} > 0$ and $T_1=T_1 (B, \omega, \eta) > 0$ are constants, and the last step follows from the tempered property of $B \in \msd$.

Note that $z(\theta_t \omega) = \sum^m_{j=1} h_j z_j (\theta_t\omega_j )$ and $h_j\in H^2(\mathbb{R}^n)\cap W^{2,p}(\mathbb{R}^n).$ Thus there is a constant $K_2 =K_2 (\omega,
\eta) \geq 1$ such that for all $r\geq K_2$,
\begin{equation}\label{446}
     \max_{1 \leq j \leq m} \int_{|x|\geq r} \left(|h_j(x)|^2+|h_j(x)|^p+|\nabla h_j(x)|^2\right) dx \leq \frac{\sigma\, \eta}{2C_8\, r_0 (\omega)},
\end{equation}
where $r_0 (\omega)$ is the tempered function in \eqref{39}. By \eqref{39} and (\ref{446}) we obtain the estimate of the last integral term in \eqref{E0},
\begin{equation}\label{E2}
    \begin{split}
    & C_8 \int^0_{-t}e^{\sigma s}\int_{|x| \geq r} \left(|\nabla z(\theta_{s}\omega)|^2+|z(\theta_{s}\omega)|^2+|z(\theta_{s}\omega)|^p \right) dx\, ds \\
    \leq &\,  C_8 \int^0_{-t} e^{\sigma s} \, \sum^m_{j=1}\int_{|x|\geq r} \left(|\nabla h_j|^2
|z_j(\theta_{s}\omega_j)|^2+|h_j|^2|z_j(\theta_{s}\omega_j)|^2+|h_j|^p |z_j (\theta_{s}\omega_j)|^p \right) dx\, ds   \\
\leq &\, \frac{\sigma \eta}{2r_0 (\omega)}\int^0_{-t} e^{\sigma s}\, \sum^m_{j=1}(|z_j(\theta_{s}\omega_j)|^2 +|z_j(\theta_{s}\omega_j)|^p)\,ds \leq \frac{\sigma \eta}{2r_0 (\omega)}\int^0_{-t} e^{\frac{1}{2} \sigma s} \, r_0 (\omega) \, ds < \eta.
    \end{split}
\end{equation}

Finally we estimate the third integral term on the right-hand side
of (\ref{E0}). Applying the Gronwall inequality to \eqref{sgre}
while taking the spatial integral over the region $r \leq |x| \leq
\sqrt{2} r$, with \eqref{sigma} in mind, we can get

\begin{equation} \label{rs2r}
    \begin{split}
     &\int_{-t}^0 e^{\sigma s} \int_{r \leq |x| \leq \sqrt{2}r} (|\nabla u(s, \omega, -t, u_{0} (\theta_{-t}\omega))|^2 + |v(s, \omega, -t, v_{0}(\theta_{-t}\omega))|^2)\, dx\, ds \\[5pt]
    \leq &\,  e^{-\sigma (s + t)} \left(\| v_0 (\theta_{-t} \omega)\|^2 + (\alpha + \delta^2 - \beta \delta) \|u_0 (\theta_{-t} \omega)\|^2 + \|\nabla u_0 (\theta_{-t} \omega)\|^2\right) \\[4pt]
    &\, + 2e^{-\sigma (s + t)} \int_{\mathbb{R}^n} \left(F(x, u_0 (\theta_{-t} \omega)) + \phi_3 (x)\right)\, dx  \\
     &\, +\, \frac{1}{\sigma} \left(C_6 + \frac{1}{\beta - \delta} \|g\|^2\right) + 2 \int_{-t}^s e^{- \sigma (s - \tau)} \Gamma_2 (\theta_\tau \omega)\, d\tau.
    \end{split}
\end{equation}
where by \eqref{446} and \eqref{39}, similar to \eqref{413} we see that for $r \geq K_2$,
\begin{gather*}
    \Gamma_2 (\theta_t \, \omega) = C_0  \int_{r \leq |x| \leq \sqrt{2}r} \left(|z(\theta_t \omega)|^2+ |\nabla z(\theta_t \omega)|^2 +|z(\theta_t \omega)|^{p} \right)\, dx   \\
    \leq C_0 \, \sum^m_{j=1}\int_{|x|\geq r} \left(|\nabla h_j|^2 |z_j(\theta_{t} \omega_j)|^2+|h_j|^2|z_j(\theta_{t} \omega_j)|^2+|h_j|^p |z_j (\theta_{t} \omega_j)|^p \right) dx \\
    \leq \frac{C_0 \, \sigma \eta}{2C_8\, r_0 (\omega)} \sum^m_{j=1} \left(|z_j(\theta_{t}\omega_j)|^2 +|z_j(\theta_{t}\omega_j)|^p\right) \leq \frac{C_0 \, \sigma \eta}{2C_8}\, e^{(\sigma/2) |t|}.
\end{gather*}
Based on \eqref{rs2r} and the above inequality as well as the
tempered property of $B\in \msd$, there exists $T_2 = T_2 (B,
\omega, \eta) > 0$ such that
\begin{equation}\label{E3}
    \begin{split}
    & \int^0_{-t} e^{\sigma s} \int_{r \leq |x| \leq \sqrt{2}r} \left(|\nabla u(s, \omega, -t, u_{0} (\theta_{-t}\omega))|^2+|v(s, \omega, -t, v_{0}(\theta_{-t}\omega))|^2 \right) dx\, ds \\
    \leq &\, C t \, e^{-\sigma t} \left[\| (u_0, v_0) (\theta_{-t} \omega)\|^2 + \|\nabla u_0 (\theta_{-t} \omega)\|^2 + \int_{\mathbb{R}^n} (F(x, u_0 (\theta_{-t} \omega))+ \phi_3 (x))\, dx\right] \\
    &\, + \frac{1}{\sigma} \left(C_6 + \frac{1}{\beta - \delta}\|g \|^2\right) + \frac{C_0\, \sigma \eta}{C_8} \int_{-t}^0  \int_{-t}^s e^{\sigma \tau -\frac{1}{2}\sigma \tau}\, d\tau \, ds \\[5pt]
    \leq &\,C t \, e^{-\sigma t} \left(\| (u_0, v_0) (\theta_{-t} \omega)\|^2 + \|\nabla u_0 (\theta_{-t} \omega)\|^2 + \|\phi_3 \|_{L^1} \right) \\[6pt]
    &\, + \frac{C}{C_2} t \, e^{-\sigma t} \left( C_1 \|u_0 (\ctw)\|_{L^p}^p + \|u_0 (\ctw)\|^2 + \|\phi_1\|^2 + \|\phi_2 \|_{L^1}\right) \\[3pt]
    &\, + \frac{1}{\sigma} \left(C_6 + \frac{1}{\beta - \delta} \|g \|^2\right) + \frac{4C_0\eta}{C_8 \sigma} \leq M, \quad  \textup{for all} \;\; t \geq T_2,
    \end{split}
\end{equation}
where the constant
$$
    M = 1 + \frac{1}{\sigma} \left(C_6 + \frac{1}{\beta - \delta} \|g\|^2 + \frac{4C_0\eta}{C_8 \sigma}\right).
$$
Now assemble all these estimates and substitute \eqref{444}, \eqref{E2} and \eqref{E3} into \eqref{E0}. It shows that for any $B \in \msd, 0 < \eta \leq 1$ and a.e. $\om \in \gw$, as long as $r\geq V = \max\, \{K_0, K_1, K_2\}$ and $t \geq \max \{T_1, T_2\}$ one has

\begin{equation} \label{Lts}
    \begin{split}
    &\, \int_{|x|\geq \sqrt{2}r} (|v(0, \omega, -t, v_{0}(\theta_{-t}\omega))|^2 + (\alpha + \delta^2 - \beta \delta)|u(0, \omega, -t, u_{0}(\theta_{-t}\omega))|^2)\, dx \\
    &\, + \int_{|x|\geq \sqrt{2}r} (|\nabla u(0, \omega, -t, u_{0}(\theta_{-t}\omega))|^2 + |u(0, \omega, -t, u_{0}(\theta_{-t}\omega))|^p)\, dx \\
    \leq &\, \int_{\mathbb{R}^n}  \rho\left(\frac{|x|^2}{r^2}\right) \left(|v(0, \omega, -t, v_{0}(\theta_{-t}\omega))|^2 +(\alpha+\delta^2-\beta\delta) |u(0, \omega, -t, u_{0}(\theta_{-t}\omega))|^2)\right) dx  \\[5pt]
    &\, + \intr \rxr |\nabla u(0, \omega, -t, u_{0}(\theta_{-t}\omega))|^2 \, dx  \\
    &\, + \frac{2}{C_3} \intr \rxr (F(x, u(0, \omega,  -t, u_{0}(\theta_{-t}\omega)) + \phi_3)\, dx \leq \left(1 + \frac{1}{C_3}\right) (2 + M)\eta.
    \end{split}
\end{equation}
By \eqref{Enorm}, the above inequality \eqref{Lts} demonstrates that for any $B \in \msd$ and a.e. $\om \in \gw$, it holds that
\begin{equation} \label{Phts}
    \begin{split}
    &\interleave \Phi (t, \ctw, B(\ctw))\interleave_{E(\mbR \backslash B_{R})} = \max_{g_0 \in B(\ctw)} \|\Phi (t, \ctw, g_0)\|_{E(\mbR \backslash B_{R})} \\
    \leq &\, \left[\left(1 + \frac{1}{C_3}\right) (2 + M)\eta\right]^{1/2} + \left[\left(1 + \frac{1}{C_3}\right) (2 + M)\eta \right]^{1/p},
    \end{split}
\end{equation}
where $R = \sqrt{2}r$. \eqref{Phts} implies (\ref{431}) according to \eqref{pbj} simply by renaming $r$ to be $R$ and $\eta$ to be $((1 + 1/C_3)(2+M)\eta)^{1/2} + ((1 + 1/C_3)(2+M)\eta)^{1/p}$.  The proof is completed.
\end{proof}

\section{\large{\bf{Pullback Asymptotic Compactness in Space $H^1 (\mbR) \times L^2 (\mbR)$}}}

In this section, we shall prove the pullback asymptotic compactness of the random dynamical sytem $\Phi$ associated with the stochastic wave equation \eqref{22} formulated into \eqref{310} and reach the existence of a random attractor for this random dynamical system. 

\begin{lemma}  \label{L51}
The following statements hold for $L^p (\mbR), \, 1 \leq p < \infty$.

\textup{1)} Let $\{\psi_m\}$ be a sequence and $\psi$ be a function
in $L^p (\mbR)$ such that $\|\psi_m - \psi \|_{L^p} \to 0$ as $m \to
\infty$. Then there exists a subsequence $\{\psi_{m_k}\}$ such that
$$
    \lim_{k\to \infty}  \psi_{m_k} (x) = \psi (x), \quad \textup{a.e. on} \;\; \mbR.
$$

\textup{2)} If a sequence $\{\psi_m\}$ and a function $\psi$ in $L^p (\mbR)$ satisfy the following two conditions\textup{:}
\begin{equation} \label{S2}
    \lim_{m\to \infty}  \psi_{m} (x) = \psi (x), \;\, \textup{a.e. on} \;\; \mbR \quad \textup{and} \quad  \psi_m \;\, \textup{is bounded in} \;\, L^p(\mbR),
\end{equation}
then  $\psi_m \to \psi$ weakly in $L^p (\mbR)$, as $m \to \infty$.

\textup{3)} For $1 < p < \infty$, if a sequence $\{\psi_m\}$ and a function $\psi$ in $L^p (\mbR)$ satisfy the following two conditions\textup{:}
\begin{equation} \label{S3}
    \lim_{m\to \infty}  \psi_{m} (x) = \psi (x), \;\, \textup{a.e. on} \;\; \mbR \quad \textup{and} \quad \lim_{m\to \infty} \|\psi_m \|_{L^p} = \|\psi \|_{L^p},
\end{equation}
then $\lim_{m\to \infty} \|\psi_m -\psi\|_{L^p} = 0$.
\end{lemma}
\begin{proof}
    Since $\mbR$ with the Lebesgue measure is a $\sigma$-finite measure space, the first item is a standard result in Real and Functional Analysis.

    For the second item, since $L^p (\mbR)$ is a reflexive Banach space for $1 < p < \infty$, the boundedness of $\{\psi_m\}$ in $L^p (\mbR)$ implies that there is $\varphi \in L^p (\mbR)$ such that $\psi_m \to \varphi$ weakly as $m \to \infty$. By Mazur's lemma, this weak convergence implies there exists a sequence $\{\zeta_m\} \subset L^p (\mbR)$ such that
\begin{equation} \;\label{Mz}
    \zeta_m \in \textup{conv} \, \{\psi_m, \psi_{m+1}, \cdots \} \;\, \textup{and} \;\, \zeta_m \to \varphi \;\, \textup{strongly in}\;\, L^p (\mbR).
\end{equation}
It follows from the condition $\psi_m \to \psi$ a.e. and $\zeta_m \in \textup{conv}\, (\bigcup_{i=m}^\infty \psi_i)$ that
\begin{equation} \label{Zs}
    \zeta_m \to \psi \;\;  \textup{a.e. in} \;\, \mbR.
\end{equation}
On the other hand, by the first statement in this lemma, the strong
convergence in \eqref{Mz} implies that there exists a subsequence
$\{\zeta_{m_k}\}$ such that $\zeta_{m_k} \to \varphi$ a.e. as $k\to
\infty$. Therefore, \eqref{Zs} leads to $\psi = \varphi$ a.e. on
$\mbR$ and $\psi_m \to \psi$ weakly as $m \to \infty$. The third
item is a known result in Functional Analysis, cf. \cite[Chapter
4]{HB}.
\end{proof}


Let us define the following energy functional on $E$: for $(u,v) \in E$,
\begin{equation} \label{511}
    Q (u, v)=\|v\|^2+ \left(\alpha+\delta^2-\beta \delta\right)\|u\|^2 + \|\nabla u\|^2+2\int_{\mathbb{R}^n} (F(x, u) + \phi_3 (x))\, dx.
\end{equation}
Compare \eqref{Enorm} and \eqref{511}, we see that
\begin{equation} \label{EQ}
    Q(u, v) \leq \|(u, v)\|_E^2 + 2\int_{\mathbb{R}^n} (F(x, u) + \phi_3 (x))\, dx.
\end{equation}

\begin{lemma} \label{L52}
    For every $\om\in\gw$ and any $B \in \msd$ and any integer $k > 0$, there exists a constant $M_1 = M_1 (B, \omega, k) >0$ such that for all $m\geq M_1$ one has $t_m > k$ and
\begin{equation}\label{Ctm}
    \begin{split}
    Q &\, (u(t, \omega, -t_m, u_{0, m}), v(t, \omega, -t_m, v_{0, m})) \leq  R(\omega)  + 1  \\[2pt]
    + &\, \frac{1}{\sigma} \, e^{\frac{1}{2} \sigma k} \left[4C_5 r_0 (\om) + C_6 + \frac{\|g\|^2}{\beta - \delta}\right], \;\;\,  t \in [-k, 0],
    \end{split}
\end{equation}
for all $(u_{0,m}, v_{0,m}) \in B(\theta_{-t_m} \om)$, where $R(\om)$ and $r_0 (\om)$ are the same as in \eqref{Kom} and \eqref{38}, respectively.
\end{lemma}

\begin{proof}
Integrate the inequality \eqref{sgre} over the time interval $[-k, t]\subset [-k, 0]$, where $\delta \geq \sigma$ by \eqref{sigma}. Similar to \eqref{420}, there exists $M_1 = M_1 (B, \om, k) > 0$ such that for all $m \geq M_1$ one has $t_m > k$ and
\begin{equation} \label{Qkt}
    \begin{split}
    &\, Q (u(t, \omega, -t_m, u_{0, m}), v(t, \omega, -t_m, v_{0, m})) \\[7pt]
    \leq &\, e^{-\sigma (t+k)} Q (u(-k, \om, -t_m, u_{0,m}), v(-k, \om, -t_m, v_{0, m})) \\
    &\, + \int_{-k}^t e^{-\sigma (t-s)} \left(2\Gamma_1 (\theta_s \om) + C_6 + \frac{\|g\|^2}{\beta - \delta}\right) ds \\
    \leq &\, R(\om) + 1 + 2C_5 \int_{-k}^t e^{-\sigma t + (\sigma - \frac{1}{2}\sigma) s} r_0 (\om)\, ds + \frac{1}{\sigma} \left(C_6 + \frac{\|g\|^2}{\beta - \delta}\right)  \\
    \leq &\, R(\om) + 1 + \frac{1}{\sigma} \left(4 e^{-\frac{1}{2} \sigma t} \,C_5 r_0 (\om) + C_6 + \frac{\|g\|^2}{\beta - \delta}\right), \quad t \in [-k, 0].
    \end{split}
\end{equation}
Therefore, \eqref{Ctm} holds.
\end{proof}

\begin{theorem}  \label{T51}
Under Assumptions \textup{I, II} and \textup{III}, for every $B=\{B(\omega)\}_{\omega\in \Omega}\in \msd$ and any sequences $t_m \rightarrow\infty$ and $g_{0,m} = (u_{0,m}, v_{0,m}) \in B(\theta_{-t_m}\omega)$, the sequence
$$
    \{\Phi(t_m, \theta_{-t_m}\omega, g_{0,m})\}^\infty_{m=1}
$$
of a pullback quasi-trajectory of the cocycle $\Phi$ associated with the problem \eqref{310} of the stochastic wave equation has a strongly convergent subsequence in $H^1 (\mbR) \times L^2 (\mbR)$.
\end{theorem}

\begin{proof}
The proof goes through the following five steps.

STEP 1.
First by Lemma \ref{L41},  there exist a constant $M_2=M_2(B, \omega)>0$ such that for all $m\geq M_2$ and $g_{0,m} \in B(\theta_{-t_m} \omega)$, we have
\begin{eqnarray}\label{Cr}
    \|\Phi(t_m, \theta_{-t_m}\omega, g_{0,m})\|_E\leq R(\omega)+1, \quad \om \in \gw,
\end{eqnarray}
where $R(\om) > 0$ is given by \eqref{Kom}. Then for any $\om\in\gw$ there is $(\tilde{u}(\om), \tilde{v} (\om))\in E$ such that, up to a subsequence relabeled the same,
\begin{equation}  \label{twc}
    \begin{split}
    &\Phi(t_m, \theta_{-t_m}\omega, g_{0,m}) \longrightarrow (\tilde{u}(\om), \tilde{v}(\om)) \quad \textup{weakly in} \;\, E; \\
    &\Phi(t_m, \theta_{-t_m}\omega, g_{0,m}) \longrightarrow (\tilde{u}(\om), \tilde{v}(\om)) \quad \textup{weakly in} \;\, H^1 (\mbR) \times L^2 (\mbR); \\
    &\Phi(t_m, \theta_{-t_m}\omega, g_{0,m}) \longrightarrow (\tilde{u}(\om), \tilde{v}(\om)) \quad \textup{weakly in} \;\, L^p (\mbR).
    \end{split}
\end{equation}
Since $E$ is a reflexive and separable Banach space, the weak lower-semicontinuity of the $E$-norm and the norm of $H^1 (\mbR) \times L^2 (\mbR)$ implies that
\begin{equation} \label{tn}
    \begin{split}
    &\liminf_{m\rightarrow \infty}\|\Phi(t_m, \theta_{-t_m}\omega, g_{0,m})\|_E\geq \|(\tilde{u}(\om), \tilde{v}(\om))\|_E, \\
    &\liminf_{m\rightarrow \infty}\|\Phi(t_m, \theta_{-t_m}\omega, g_{0,m})\|_{H^1 \times L^2} \geq \|(\tilde{u}(\om), \tilde{v}(\om))\|_{H^1 \times L^2}.
    \end{split}
\end{equation}

Next we shall prove that in the Hilbert space $H^1 (\mbR) \times L^2
(\mbR)$,
\begin{equation} \label{Strc}
    \Phi(t_m, \theta_{-t_m}\omega, g_{0,m}) \longrightarrow (\tilde{u}(\om), \tilde{v}(\om)) \quad \textup{strongly}.
\end{equation}
It suffices to show that
\begin{eqnarray}\label{54}
    \limsup_{m\rightarrow \infty}\|\Phi(t_m, \theta_{-t_m}\omega, g_{0,m})\|_{H^1 \times L^2}\leq \|(\tilde{u}(\om), \tilde{v}(\om))\|_{H^1 \times L^2}.
\end{eqnarray}
Then \eqref{tn} and \eqref{54} lead to
$$
    \lim_{m\rightarrow \infty}\|\Phi(t_m, \theta_{-t_m}\omega, g_{0,m})\|_{H^1 \times L^2} = \|(\tilde{u}(\om), \tilde{v}(\om))\|_{H^1 \times L^2}.
$$
By the item 3 of Lemma \ref{L51}, we shall obtain \eqref{Strc}.

STEP 2. By Lemma \ref{L52} and \eqref{35}, we can deduce that there exists a constant $C > 0$ such that for every $\om\in\gw$ and any given integer $k > 0$, whenever $m\geq M_1$ one has $t_m > k$ and
\begin{equation}\label{ktm}
    \begin{split}
    &\, \|(u(t, \omega, -t_m, u_{0, m}), v(t, \omega, -t_m, v_{0, m}))\|_E   \\
    \leq &\, C\left[ R(\omega)+1 + \frac{1}{\sigma} \, e^{\frac{1}{2} \sigma k} \left(4C_5 r_0 (\om) + C_6 + \frac{\|g\|^2}{\beta - \delta}\right)\right]^{1/2}  \\
    + &\, C\left[ R(\omega)+1 + \frac{1}{\sigma} \, e^{\frac{1}{2} \sigma k} \left(4C_5 r_0 (\om) + C_6 + \frac{\|g\|^2}{\beta - \delta}\right)\right]^{1/p}, \;\;  t \in [-k, 0],
    \end{split}
\end{equation}
for any $(u_{0,m}, v_{0,m}) \in B(\theta_{-t_m} \om)$. In particular, \eqref{ktm} is satisfied for $t = -k$.

Then by the Banach-Alaoglu theorem, there exists a sequence
$\{\tilde{u}_k (\om), \tilde{v}_k (\om)\}^\infty_{k=1}$ in the space
$E$ and subsequences of $\{-t_m\}_{m=1}^\infty$ and $\{(u_{0, m},
v_{0, m})\}^\infty_{m=1}$ again relabeled as the same, such that for
all $\om \in \gw$ and every integer $k \geq 1$,
\begin{eqnarray} \label{ktmc}
    (u(-k, \omega, -t_m, u_{0, m}), v(-k, \omega, -t_m, v_{0, m}))\longrightarrow (\tilde{u}_k (\om), \tilde{v}_k (\om))  \;\; \textup{weakly in} \;\, E,
\end{eqnarray}
as $m\rightarrow \infty$, which can be extracted through a diagonal selection procedure.

By the weakly continuous dependence on the initial data of the solutions stated in Lemma \ref{Lwks}, the weak convergence \eqref{ktmc} and the fact of concatenation,
\begin{equation} \label{0m}
    \begin{split}
    &(u(0, \omega, -t_m, u_{0, m}), \;\, v(0, \omega, -t_m, v_{0, m}))  \\[5pt]
    = (u(0, \omega, -k, &\, u(-k, \omega, -t_m, u_{0, m})), \;\, v(0, \omega, -k, v(-k, \omega, -t_m, v_{0, m}))),
    \end{split}
\end{equation}
imply that for all integers $k \geq 1$ and $\om\in\gw$, when $m\rightarrow \infty,$
\begin{equation}\label{59}
    (u(0, \omega, -t_m, u_{0, m}), v(0, \omega, -t_m, v_{0, m}))\longrightarrow (u(0, \omega, -k, \tilde{u}_k), v(0, \omega, -k, \tilde{v}_k))
\end{equation}
weakly in $E$. By \eqref{pbj}, \eqref{twc} and \eqref{59} we reach the following equality that for every $\om \in \gw$ and all positive integers $k$,
\begin{equation}\label{510}
    (\tilde{u}(\om), \tilde{v}(\om))=(u(0, \omega, -k, \tilde{u}_k (\om)), v(0, \omega, -k, \tilde{v}_k (\om))).
\end{equation}

By the equalities \eqref{42}-\eqref{47}, the weak solutions $(u, v)$ 
of \eqref{310} satisfies
\begin{equation}\label{512}
    \begin{split}
    \frac{d}{dt} &\, Q(u, v)+2\sigma Q(u, v) = -2(\beta-\delta-\sigma)\|v\|^2-2(\delta-\sigma)\left(\alpha+\delta^2 - \beta\delta\right)\|u\|^2  \\
    &\, -2(\delta-\sigma)\|\nabla u\|^2+4\sigma\int_{\mathbb{R}^n} (F(x, u) + \phi_3 (x))\, dx -2\delta \inpt{f(x, u), u}  \\[4pt]
&\, +2\varepsilon\left(\alpha+\delta^2 - \beta\delta\right)\inpt{z(\theta_t\omega), u}+2\varepsilon\inpt{\nabla z(\theta_t\omega), \nabla u} +2\varepsilon\inpt{z(\theta_t\omega), f(x,u)}  \\[10pt]
    &\, +(4\delta\varepsilon - 2\beta\varepsilon)\inpt{z(\theta_t\omega), v}+2\inpt{g, v} := G (u(t, \omega, \tau, u_0), v(t, \omega, \tau, v_0)).
    \end{split}
\end{equation}
From \eqref{510} and \eqref{512}, for any integer $k \geq 1$ we have
\begin{equation} \label{Quvk}
    \begin{split}
    Q(\tilde{u}(\om), \tilde{v}(\om)) &= e^{-2\sigma k} Q (\tilde{u}_k (\om), \tilde{v}_k (\om)) \\[3pt]
    &\, +\int^0_{-k} e^{2\sigma \xi} \, G (u(\xi, \omega, -k, \tilde{u}_k), v(\xi, \omega, -k, \tilde{v}_k ))\, d\xi.
    \end{split}
\end{equation}

STEP 3. From the concatenation \eqref{0m} and \eqref{512}, on the other hand, we obtain

\begin{equation}\label{longQ}
    \begin{split}
    &\, Q(u(0, \omega, -t_m, u_{0, m}),  v(0, \omega, -t_m, v_{0, m}))  \\[8pt]
    \leq &\,\, e^{-2\sigma k} Q (u(-k, \omega, -t_m, u_{0, m}),  v(-k, \omega, -t_m, v_{0, m})) \\[6pt]
    & -2(\beta-\delta-\sigma)\int^0_{-k}e^{2\sigma \xi}\|v(\xi, \omega, -k, v(-k, \omega, -t_m, v_{0, m}))\|^2d\,\xi    \\
    &-2(\delta-\sigma)\left(\alpha+\delta^2-\beta\delta\right)\int^0_{-k}e^{2\sigma \xi}\|u(\xi, \omega, -k, u(-k, \omega, -t_m, u_{0, m}))\|^2d\xi     \\
    & -2(\delta-\sigma)\int^0_{-k}e^{2\sigma \xi}\|\nabla u(\xi, \omega, -k, u(-k, \omega, -t_m, u_{0, m}))\|^2 \, d\xi   \\
    &+4\sigma\int^0_{-k}e^{2\sigma \xi}\int_{\mathbb{R}^n} (F(x, u(\xi, \omega, -k, u(-k, \omega, -t_m, u_{0, m}))) + \phi_3 (x))\, dx\,d\xi   \\
    &-2\delta\int^0_{-k}e^{2\sigma \xi}\int_{\mathbb{R}^n}f(x, u(\xi, \omega, -k, u(-k, \omega, -t_m, u_{0, m}))) \cdot  \\[6pt]
    &\, \cdot u(\xi, \omega, -k, u(-k, \omega, -t_m, u_{0, m}))\, dx\, d\xi   \\[6pt]
    &+2\varepsilon\left(\alpha+\delta^2 - \beta\delta\right)\int^0_{-k}e^{2\sigma \xi}\int_{\mathbb{R}^n}z(\theta_\xi\omega)u(\xi, \omega, -k, u(-k, \omega, -t_m, u_{0, m}))\, dx\,d\xi  \\[4pt]
    &+2\varepsilon\int^0_{-k}e^{2\sigma \xi}\int_{\mathbb{R}^n} \nabla z(\theta_\xi\omega) \nabla u(\xi, \omega, -k, u(-k, \omega, -t_m, u_{0, m}))\, dx\,d\xi   \\
    &+2\varepsilon\int^0_{-k}e^{2\sigma \xi}\int_{\mathbb{R}^n}  z(\theta_\xi\omega) f(x, u(\xi, \omega, -k, u(-k, \omega, -t_m, u_{0, m})))\, dx\,d\xi   \\
    &+(4\delta\varepsilon - 2\beta\varepsilon)\int^0_{-k}e^{2\sigma \xi}\int_{\mathbb{R}^n} z(\theta_\xi\omega) v(\xi, \omega, -k, v(-k, \omega, -t_m, v_{0, m}))\, dx\,d\xi   \\
    & +2\int^0_{-k}e^{2\sigma \xi}\int_{\mathbb{R}^n} g(x)\, v(\xi, \omega, -k, v(-k, \omega, -t_m, v_{0, m}))\, dx\, d\xi.
    \end{split}
\end{equation}

Below we treat all the terms on the right-hand side of \eqref{longQ}.

1) For the first term on the right-hand side of \eqref{longQ}, by \eqref{Ctm} in Lemma \ref{L52}, for every $\om \in \gw$ and all $m\geq M_1 (B, \om, k)$ we have
\begin{equation}\label{516}
    \begin{split}
    &\, e^{-2\sigma k} Q(u(-k, \omega, -t_m, u_{0, m}),  v(-k, \omega, -t_m, v_{0, m}))  \\[6pt]
    \leq &\, e^{-2\sigma k} \left(R(\om) + 1 + \frac{1}{\sigma} \, e^{\frac{1}{2} \sigma k} \left[4C_5 r_0 (\om) + C_6 + \frac{\|g\|^2}{\beta - \delta}\right]\right)  \\
    \leq &\, e^{-\sigma k} \left(R(\om) + 1 + \frac{1}{\sigma} \left[4C_5 r_0 (\om) + C_6 + \frac{\|g\|^2}{\beta - \delta}\right]\right).
    \end{split}
\end{equation}

2) For the second term on the right-hand side of \eqref{longQ}, by \eqref{ktmc} and the weakly continuous dependence of solutions on the initial data stated in Lemma \ref{Lwks}, we find that for every
$\om \in \gw$ and all $\xi\in [-k, 0],$ when $m\rightarrow \infty$,
$$
v(\xi, \omega, -k, v(-k, \omega, -t_m, v_{0, m}))\longrightarrow  v(\xi, \omega, -k, \tilde{v}_{k}(\om)) \quad  \textup{weakly in} \;\, L^2(\mathbb{R}^n),
$$
which implies that for all $\xi\in [-k, 0]$,
\begin{equation}\label{517}
    \liminf_{m\rightarrow \infty}\|v(\xi, \omega, -k, v(-k, \omega, -t_m, v_{0, m}))\|^2 \geq \|v(\xi, \omega, -k, \tilde{v}_{k}(\om))\|^2.
\end{equation}
By \eqref{517} and Fatou's lemma we obtain
\begin{equation}\label{518}
    \begin{split}
    &\liminf_{m\rightarrow \infty}\, \int^0_{-k} e^{2\sigma \xi}\|v(\xi, \omega, -k, v(-k, \omega, -t_m, v_{0, m}))\|^2\, d\xi  \\
    \geq &\, \int^0_{-k} e^{2\sigma \xi}\liminf_{m\rightarrow \infty}\, \|v(\xi, \omega, -k, v(-k, \omega, -t_m, v_{0, m}))\|^2\, d\xi   \\
    \geq &\, \int^0_{-k} e^{2\sigma \xi}\|v(\xi, \omega, -k, \tilde{v}_{k}(\om))\|^2\, d\xi.
    \end{split}
\end{equation}
Therefore, since \eqref{40} and \eqref{sigma} implies $\beta - \delta - \sigma \geq \beta - 2\delta > 0$, \eqref{518} leads to
\begin{equation}\label{519}
    \begin{split}
    &\limsup_{m\rightarrow \infty}\, -2(\beta -\delta-\sigma)\int^0_{-k} e^{2\sigma \xi}\|v(\xi, \omega, -k, v(-k, \omega, -t_m, v_{0, m}))\|^2\, d\xi  \\
    =&\, -2(\beta -\delta-\sigma)\liminf_{m\rightarrow \infty}\, \int^0_{-k} e^{2\sigma \xi}\|v(\xi, \omega, -k, v(-k, \omega, -t_m, v_{0, m}))\|^2\, d\xi  \\
    \leq &\,  -2(\beta -\delta-\sigma)\int^0_{-k} e^{2\sigma \xi}\|v(\xi, \omega, -k, \tilde{v}_{k}(\om))\|^2\, d\xi.
    \end{split}
\end{equation}
Similarly for the third and fourth terms, by \eqref{ktmc} and Fatou's lemma we obtain
\begin{equation}\label{520}
    \begin{split}
    \limsup_{m\rightarrow \infty} &\, -2(\delta-\sigma)\left(\alpha+\delta^2-\beta\delta\right)\int^0_{-k} e^{2\sigma \xi}\|u(\xi, \omega, -k, u(-k, \omega, -t_m, u_{0, m}))\|^2\, d\xi  \\
    &\leq -2(\delta-\sigma)\left(\alpha+\delta^2-\beta\delta\right)\int^0_{-k} e^{2\sigma \xi}\|u(\xi, \omega, -k, \tilde{u}_{k}(\om))\|^2\, d\xi ,   \\
    \limsup_{m\rightarrow \infty} &\, -2(\delta-\sigma)\int^0_{-k} e^{2\sigma \xi}\|\nabla u(\xi, \omega, -k, u(-k, \omega, -t_m, u_{0, m}))\|^2\, d\xi   \\
    &\leq -2(\delta-\sigma)\int^0_{-k} e^{2\sigma \xi}\|\nabla u(\xi, \omega, -k, \tilde{u}_{k}(\om))\|^2\, d\xi.
    \end{split}
\end{equation}

3) For the fifth term on the right-hand side of \eqref{longQ}, we have
\begin{equation} \label{Ftm}
    \begin{split}
        &\left|\int^0_{-k} e^{2\sigma \xi}\int_{\mathbb{R}^n} \left(F(x, u(\xi,\omega, -k, u(-k, \omega, -t_m, u_{0, m})))-F(x, u(\xi, \omega, -k, \tilde{u}_k))\right) dx\,d\xi \right|  \\
       \leq& \int^0_{-k} e^{2\sigma \xi}\int_{|x|> r} \left|F(x, u(\xi, \omega, -k, u(-k, \omega, -t_m, u_{0, m})))-F(x, u(\xi, \omega, -k, \tilde{u}_k))\right|dx\,d\xi   \\
       +& \int^0_{-k} e^{2\sigma \xi}\int_{|x|\leq r} \left|F(x, u(\xi, \omega, -k, u(-k, \omega, -t_m, u_{0, m})))-F(x, u(\xi, \omega, -k, \tilde{u}_k))\right|dx\,d\xi.
    \end{split}
\end{equation}

A) For any given $\eta>0,$ by the proof of Lemma \ref{L42} adapted to the time interval $(-\infty, -k]$, there exist $M_3=M_3 (B, \omega, \eta) >M_2$ and $K =K (B, \omega, \eta)\geq 1$ such that for $\om\in\gw$ and $\xi\in [-k, 0]$, whenever $r \geq K$ and $m \geq M_3$ one has
\begin{eqnarray}\label{523}
    \int_{|x| > r} \left(|u(\xi,  \omega,  -t_m, u_{0, m})|^2 + |u(\xi,  \omega,  -t_m, u_{0, m})|^p + |\phi_1|^2 + |\phi_2| + |\phi_3|\right) dx < \eta.
\end{eqnarray}
In view of the Assumption II, there exists a constant $L_1 > 0$ such that for all $\om\in\gw$ and $\xi\in [-k, 0]$, one has
\begin{equation*}  
    \begin{split}
    &\int_{|x|> r} |F(x, u(\xi, \omega, -t_m,  u_{0, m}))|\, dx  \\
    \leq &\, \int_{|x| > r} L_1 (|u(\xi,  \omega,  -t_m, u_{0, m})|^2 + |u(\xi,  \omega,  -t_m, u_{0, m})|^p + |\phi_1|^2 + |\phi_2| + |\phi_3|) dx  \\[3pt]
    < &\,\, L_1 \eta,  \qquad \textup{for all} \;\, r \geq K, \;  m \geq M_3.
    \end{split}
\end{equation*}

B) Since \eqref{ktmc} shows that
$$
    \tilde{u}_k (\om) = \textup{(weak)} \lim_{m\to \infty} u (-k, \om, -t_m, u_{0,m}) \quad \textup{in} \;\,\;  L^2 (\mbR)\cap L^p (\mbR),
$$
by the weakly continuous dependence of solutions on intial data stated in Lemma \ref{Lwks} and the weak lower-semicontinuity of the $L^2$ and $L^p$ norms, it follows from \eqref{523} that
\begin{equation*}
    \begin{split}
    &\int_{-k}^0 e^{2\sigma \xi} \int_{|x|> r} |F(x, u(\xi, \om, -k, \tilde{u}_k))|\, dx\, d\xi  \\
    \leq &\, \int_{-k}^0 e^{2\sigma \xi} \int_{|x|> r} L_1 (|u(\xi, \om, -k, \tilde{u}_k)|^2 + |u(\xi, \om, -k, \tilde{u}_k)|^p + |\phi_1|^2 + |\phi_2| + |\phi_3|)\, dx\, d\xi  \\
    = &\, \int_{-k}^0 e^{2\sigma \xi} L_1 \left(\|u(\xi, \om, -k, \tilde{u}_k)\|_{L^2 (\mbR \backslash B_r)}^2 + \|u(\xi, \om, -k, \tilde{u}_k)\|_{L^p (\mbR \backslash B_r)}^p \right) d\xi  \\
    &\, + \int_{-k}^0 e^{2\sigma \xi} L_1 \int_{|x|>r} (|\phi_1|^2 + |\phi_2| + |\phi_3|)\, dx\, d\xi   \\
    \leq &\, \int_{-k}^0 e^{2\sigma \xi} L_1 \left[\liminf_{m\to\infty} \|u(\xi, \om, -k, \tilde{u}_k)\|_{L^2 (\mbR \backslash B_r)}^2 + \liminf_{m\to\infty} \|u(\xi, \om, -k, \tilde{u}_k)\|_{L^p (\mbR \backslash B_r)}^p \right] d\xi  \\
    &\, + \int_{-k}^0 e^{2\sigma \xi} L_1 \int_{|x|>r} (|\phi_1|^2 + |\phi_2| + |\phi_3|)\, dx\, d\xi \leq \frac{L_1}{2\sigma} \eta, \quad \textup{for} \;\; \om\in\gw, r \geq  K_1, m \geq M_3.
    \end{split}
\end{equation*}

The above two inequalities show that there exists a constant $L_2 = L_1 (1 + 1/(2\sigma)) > 0$ such that the first term on the right-hand side of \eqref{Ftm} satisfies
\begin{equation} \label{528}
    \begin{split}
        &\int^0_{-k}e^{2\sigma \xi}\int_{|x|> r} |F(x, u(\xi, \omega, -k, u(-k, \omega, -t_m, u_{0, m})))-F(x, u(\xi, \omega, -k, \tilde{u}_k))|\, dx\,d\xi  \\
       \leq &\int^0_{-k}e^{2\sigma \xi}\int_{|x|> r} \left(|F(x, u(\xi, \omega, -t_m, u_{0, m}))| + |F(x, u(\xi, \omega, -k, \tilde{u}_{k}))|\right) dx\,d\xi \leq L_2 \eta,
    \end{split}
\end{equation}
for all $\om\in\gw, r\geq K$ and $m\geq M_3$.

C) For the second term on the right-hand side of \eqref{Ftm}, by \eqref{ktmc} we have
\begin{equation*}  \label{wcv}
    u(\xi, \omega, -k, u(-k, \omega, -t_m, u_{0, m})) \longrightarrow u(\xi, \omega, -k, \tilde{u}_k) \;\; \mathrm{weakly } \;\, \mathrm{in} \;\, H^1(\mathbb{B}_r) \cap L^p(\mathbb{B}_r).
\end{equation*}
Since $H^1(\mathbb{B}_r)$ is compactly embedded in $L^2(\mathbb{B}_r)$, it follows that for any $\om\in \gw$ and $\xi \in [-k,0]$,
\begin{equation} \label{4str}
    u(\xi, \omega, -k, u(-k, \omega, -t_m, u_{0, m})) \longrightarrow u(\xi, \omega, -k, \tilde{u}_k) \;\, \mathrm{strongly } \;\, \mathrm{in} \;\, L^2(\mathbb{B}_r).
\end{equation}
Then by the first item of Lemma \ref{L51} and the continuity of $F(x,u)$, as $m \to \infty$,
\begin{equation} \label{Fae}
    F(x, u(\xi, \omega, -k, u(-k, \omega, -t_m, u_{0, m}))) \longrightarrow F(x, u(\xi, \omega, -k, \tilde{u}_k)) \;\; \textup{in} \;\, \mathbb{B}_r.
\end{equation}
On the other hand, by the Assumption II and Lemma \ref{L52}, we have a uniform bound that there exists a constant $L_3 > 0$ such that
\begin{equation} \label{Fxu}
    \begin{split}
    &\int_{|x|< r} \left|F(x, u(\xi, \omega, -k, u(-k, \omega, -t_m, u_{0, m})))\right| \,dx  \\
    \leq &\, L_1 \left(\|u (\xi, \omega, -k, u(-k, \omega, -t_m, u_{0, m}))\|_{L^2 (B_r)}^2 \right.  \\[3pt]
    &\, \left. + \|u (\xi, \omega, -k, u(-k, \omega, -t_m, u_{0, m}))\|_{L^p (B_r)}^p + \|\phi_1\|^2 + \|\phi_2 \|_{L^1(\mbR)} + \|\phi_3\|_{L^1 (\mbR)}\right)  \\
    \leq &\, L_3 \left[R(\om) + 1 + \frac{1}{\sigma} \, e^{\frac{1}{2} \sigma k} \left(4C_5 r_0 (\om) + C_6 + \frac{\|g\|^2}{\beta - \delta}\right)+ \|\phi_1\|^2 + \|\phi_2 \|_{L^1} + \|\phi_3\|_{L^1}\right]
    \end{split}
\end{equation}
for all $\om\in\gw, \, \xi \in [-k, 0]$ and $m \geq M_1$. By the second item of Lemma \ref{L52}, it follows from \eqref{Fae} and \eqref{Fxu} that
$$
    F(x, u(\xi, \omega, -k, u(-k, \omega, -t_m, u_{0, m}))) \longrightarrow F(x, u(\xi, \omega, -k, \tilde{u}_k)) \;\; \textup{weakly in} \;\, L^1 (\mathbb{B}_r),
$$
as $m \to \infty$. Consequently, when $m \to \infty$,
\begin{equation} \label{Fic}
    \int_{|x|< r} F(x, u(\xi, \omega, -k, u(-k, \omega, -t_m, u_{0, m})))dx \longrightarrow \int_{|x|< r} F(x, u(\xi, \omega, -k, \tilde{u}_k))\,
    dx.
\end{equation}
Furthermore, by \eqref{Fxu} we have
\begin{equation} \label{Fbd}
    \begin{split}
    &\left|\int_{|x|< r} \left[F(x, u(\xi, \omega, -k, u(-k, \omega, -t_m, u_{0, m})))-F(x, u(\xi, \omega, -k, \tilde{u}_k)) \right] dx\right|  \\
    \leq &\, L_3 \left[R(\om) + 1 + \frac{1}{\sigma} \, e^{\frac{1}{2} \sigma k} \left(4C_5 r_0 (\om) + C_6 + \frac{\|g\|^2}{\beta - \delta}\right)+ \|\phi_1\|^2 + \|\phi_2 \|_{L^1} + \|\phi_3\|_{L^1}\right] \\[3pt]
    &\, +\|F(\cdot, u(\xi, \omega, -k, \tilde{u}_k))\|_{L^1 (\mbR)}.
    \end{split}
\end{equation}
According to the Lebesgue dominated convergence theorem,  \eqref{Fic} and \eqref{Fbd} imply that for every $\om\in\gw$, integer $k \geq 1$ and any given $r \geq K$,
\begin{equation}\label{Frc}
    \begin{split}
    \lim_{m\to \infty} \int^0_{-k} &\, e^{2\sigma \xi} \int_{|x|<r} F(x, u(\xi, \omega, -k, u(-k, \omega, -t_m, u_{0, m})))\, dx\, d\xi  \\
    & = \int^0_{-k}e^{2\sigma \xi}\int_{|x|< r} F(x, u(\xi, \omega, -k, \tilde{u}_k))\, dx\, d\xi.
    \end{split}
\end{equation}
Combine \eqref{Ftm}, \eqref{528} and \eqref{Frc}, we obtain
\begin{equation} \label{Flim}
    \begin{split}
    \lim_{m\rightarrow \infty}\int^0_{-k}&\, e^{2\sigma \xi}\int_{\mathbb{R}^n} \left(F(x, u(\xi, \omega, -k, u(-k, \omega, -t_m, u_{0, m}))) + \phi_3 (x)\right)\, dx\, d\xi  \\
    & = \int^0_{-k} \, e^{2\sigma \xi}\int_{\mathbb{R}^n} \left(F(x, u(\xi, \omega, -k, \tilde{u}_k)) + \phi_3 (x)\right)\, dx\, d\xi.
    \end{split}
\end{equation}

4) By an argument similar to the proof of \eqref{Flim} shown above, we can also prove the convergence of the sixth term on the right-hand side of \eqref{longQ}. Namely,
\begin{equation}  \label{6tm}
    \begin{split}
    \lim_{m\to \infty} &\, \int^0_{-k} e^{2\sigma \xi} \int_{\mathbb{R}^n}f(x, u(\xi, \omega, -k, u(-k, \omega, -t_m, u_{0, m}))) \cdot  \\[6pt]
    &\, \cdot u(\xi, \omega, -k, u(-k, \omega, -t_m, u_{0, m}))\, dx\, d\xi  \\[4pt]
    &\, = \int^0_{-k} e^{2\sigma \xi}\int_{\mathbb{R}^n}f(x, u(\xi, \omega, -k, \tilde{u}_{k}(\om))) u(\xi, \omega, -k, \tilde{u}_{k}(\om))\, dx\, d\xi.
    \end{split}
\end{equation}

5) The convergence of the remaining terms on the right-hand side of \eqref{longQ} can be shown even simpler:
\begin{align*}
    \lim_{m\to \infty} \int^0_{-k} & e^{2\sigma \xi}\int_{\mathbb{R}^n}z(\theta_\xi\omega)u(\xi, \omega, -k, u(-k, \omega, -t_m, u_{0, m}))\, dx\, d\xi  \\
    &=\int^0_{-k}e^{2\sigma \xi}\int_{\mathbb{R}^n}z(\theta_\xi\omega)u(\xi, \omega, -k, \tilde{u}_k (\om))\, dx\, d\xi, \\
    \lim_{m\to \infty} \int^0_{-k} & e^{2\sigma \xi}\int_{\mathbb{R}^n} \nabla z(\theta_\xi\omega) \nabla u(\xi, \omega, -k, u(-k, \omega, -t_m, u_{0, m}))\, dx\, d\xi  \\
    &= \int^0_{-k}e^{2\sigma \xi}\int_{\mathbb{R}^n}\nabla z(\theta_\xi\omega) \nabla u(\xi, \omega, -k, \tilde{u}_{k} (\om))\, dx\, d\xi, \\
    \lim_{m\to \infty}\int^0_{-k} & e^{2\sigma \xi}\int_{\mathbb{R}^n} z(\theta_\xi\omega) f(x, u(\xi, \omega, -k, u(-k, \omega, -t_m, u_{0, m})))\, dx\, d\xi  \\
    &= \int^0_{-k} e^{2\sigma \xi}\int_{\mathbb{R}^n} z(\theta_\xi\omega) f(x, u(\xi, \omega, -k, \tilde{u}_{k}(\om)))\, dx\, d\xi,  \\
    \lim_{m\to \infty} \int^0_{-k} & e^{2\sigma \xi}\int_{\mathbb{R}^n} z(\theta_\xi\omega) v(\xi, \omega, -k, v(-k, \omega, -t_m, v_{0, m}))\, dx\, d\xi  \\
    &=\int^0_{-k} e^{2\sigma \xi}\int_{\mathbb{R}^n}  z(\theta_\xi\omega) v(\xi, \omega, -k, \tilde{v}_{k}(\om))\, dx\, d\xi,  \\
    \lim_{m\to \infty} \int^0_{-k} & e^{2\sigma \xi}\int_{\mathbb{R}^n} g(x) \, v(\xi, \omega, -k, v(-k, \omega, -t_m, v_{0, m}))\, dx\, d\xi  \\
    &= \int^0_{-k} e^{2\sigma \xi}\int_{\mathbb{R}^n} g(x)\, v(\xi, \omega, -n, \tilde{v}_{k}(\om))\, dx\, d\xi.
\end{align*}

STEP 4.
Take the limit of \eqref{longQ} as $m \to \infty$ and assemble together the results shown above in the items 1) through 5) of Step 3. Then we get

\begin{equation}\label{LQ}
    \begin{split}
    &\limsup_{m\rightarrow \infty} \, Q(u(0, \omega, -t_m, u_{0, m}),  v(0, \omega, -t_m, v_{0, m}))  \\
    \leq &\, e^{-\sigma k} \left(R(\om) + 1 + \frac{1}{\sigma} \left[4C_5 r_0 (\om) + C_6 + \frac{\|g\|^2}{\beta - \delta}\right]\right)   \\
    &\, -2(\beta -\delta-\sigma)\int^0_{-k} e^{2\sigma \xi}\|v(\xi, \omega, -k, \tilde{v}_{k}(\om))\|^2\, d\xi   \\
    &-2(\delta-\sigma) \left(\alpha+\delta^2-\beta\delta\right)\int^0_{-k} e^{2\sigma \xi}\|u(\xi, \omega, -k, \tilde{u}_{k}(\om))\|^2\, d\xi   \\
    &- 2(\delta-\sigma) \int^0_{-k} e^{2\sigma \xi}\|\nabla u(\xi, \omega, -k, \tilde{u}_{k}(\om))\|^2\, d\xi   \\
    &+ 4\sigma \int^0_{-k} e^{2\sigma \xi}\int_{\mathbb{R}^n} (F(x, u(\xi, \omega, -k, \tilde{u}_k (\om))) + \phi_3 (x))\, dx\,d\xi  \\
    &-2\delta\int^0_{-k} e^{2\sigma \xi}\int_{\mathbb{R}^n} f(x, u(\xi, \omega, -k, \tilde{u}_{k}(\om))) u(\xi, \omega, -k, \tilde{u}_{k}(\om))\, dx\,d\xi  \\
    &+2\varepsilon\left(\alpha+\delta^2 - \beta\delta\right) \int^0_{-k} e^{2\sigma \xi}\int_{\mathbb{R}^n}z(\theta_\xi\omega)u(\xi, \omega, -k, \tilde{u}_k (\om))\, dx\,d\xi  \\
    &+2\varepsilon\int^0_{-k} e^{2\sigma \xi}\int_{\mathbb{R}^n}\nabla z(\theta_\xi\omega) \nabla u(\xi, \omega, -k, \tilde{u}_{k}(\om))\, dx\,d\xi  \\
    &+2\varepsilon\int^0_{-k} e^{2\sigma \xi}\int_{\mathbb{R}^n}  z(\theta_\xi\omega) f(x, u(\xi, \omega, -k, \tilde{u}_{k}(\om)))\, dx\, d\xi  \\
    &+(4\delta\varepsilon-2\beta\varepsilon)\int^0_{-k} e^{2\sigma \xi}\int_{\mathbb{R}^n}  z(\theta_\xi\omega) v(\xi, \omega, -k, \tilde{v}_{k}(\om))\, dx\,d\xi   \\
    &+2\int^0_{-k} e^{2\sigma \xi}\int_{\mathbb{R}^n} g(x)\, v(\xi, \omega, -k, \tilde{v}_{k}(\om)) \,dx\,d\xi.
    \end{split}
\end{equation}
It follows from \eqref{Quvk} and \eqref{LQ} that
\begin{equation} \label{Qk}
    \begin{split}
    &\limsup_{m\to \infty}\, Q (u(0, \omega, -t_m, u_{0, m}),  v(0, \omega, -t_m, v_{0, m}))  \\
    \leq&\, e^{-\sigma k} \left(R(\om) + 1 + \frac{1}{\sigma} \left[4C_5 r_0 (\om) + C_6 + \frac{\|g\|^2}{\beta - \delta}\right]\right)   \\
    &\, +\int^0_{-k} e^{2\sigma \xi}\, G(u(\xi, \omega, -k, \tilde{u}_k), v(\xi, \omega, -k, \tilde{v}_k))\, d\xi,  \\
    =&\, e^{-\sigma k} \left(R(\om) + 1 + \frac{1}{\sigma} \left[4C_5 r_0 (\om) + C_6 + \frac{\|g\|^2}{\beta - \delta}\right]\right)   \\
    &\, + Q (\tilde{u}(\om), \tilde{v}(\om)) - e^{-2\sigma k} Q(\tilde{u}_k (\om), \tilde{v}_k (\om))  \\[4pt]
    \leq&\, e^{-\sigma k} \left(R(\om) + 1 + \frac{1}{\sigma} \left[4C_5 r_0 (\om) + C_6 + \frac{\|g\|^2}{\beta - \delta}\right]\right) +Q(\tilde{u}(\om), \tilde{v}(\om)),
    \end{split}
\end{equation}
because $- e^{-2\sigma k} Q(\tilde{u}_k (\om), \tilde{v}_k (\om)) \leq 0$ due to \eqref{35} and \eqref{511}. Take limit $k\rightarrow \infty$ to obtain

\begin{eqnarray} \label{Qlim}
    \limsup_{m\rightarrow \infty} \, Q\,(u(0, \omega, -t_m, u_{0, m}),  v(0, \omega, -t_m, v_{0, m})) \leq Q\,(\tilde{u}(\om), \tilde{v}(\om)).
\end{eqnarray}
On the other hand, from \eqref{510}, \eqref{Fae} and \eqref{Fxu} it follows that
\begin{equation}  \label{FF}
    \lim_{m\to \infty} \int_{\mathbb{R}^n}F(x, u(0, \omega, -t_m, u_{0, m}))\, dx =  \int_{\mathbb{R}^n}F(x, \tilde{u})\, dx,
\end{equation}
which along with \eqref{Qlim} shows that
\begin{equation}  \label{Nm}
    \begin{split}
    &\limsup_{m\to \infty} \, \left(\|v(0, \omega, -t_m, v_{0, m})\|^2 + (\alpha+\delta^2-\beta \delta) \|u(0, \omega, -t_m, u_{0, m})\|^2 \right.  \\[3pt]
    &\left. +\, \|\nabla u(0, \omega, -t_m, u_{0, m})\|^2 \right) \leq \|\tilde{v}\|^2 +(\alpha+\delta^2 - \beta \delta) \|\tilde{u}\|^2 +\|\nabla \tilde{u}\|^2.
    \end{split}
\end{equation}

STEP 5. Note that the norm of $H^1 (\mbR) \times L^2 (\mbR)$ is equivalent to
$$
    \|(u,v)\|_\Pi \overset{\textup{def}}{=} Q(u,v) - 2 \int_{\mbR} (F(x,u) + \phi_3 (x))\, dx = \|v\|^2 + (\alpha + \delta^2 - \beta \delta) \|u\|^2 + \|\nabla
    u\|^2.
$$
Same as the second inequality in \eqref{tn}, from the weak convergence shown by \eqref{twc}, for any $g_{0,m} = (u_{0,m}, v_{0,m}) \in B(\theta_{-t_m} \om)$ we have
\begin{equation*}  
    \liminf_{m\to \infty} \, \| \Phi (t_m, \theta_{-t_m} \om, g_{0,m} )\|_\Pi  \geq \|(\tilde{u} (\om), \tilde{v}(\om))\|_\Pi.
\end{equation*}
Meanwhile, \eqref{Nm} implies that
\begin{equation*} 
    \limsup_{m\to \infty} \, \| \Phi (t_m, \theta_{-t_m} \om, g_{0,m})\|_\Pi  \leq \|(\tilde{u} (\om), \tilde{v}(\om))\|_\Pi.
\end{equation*}
It follows that
\begin{equation} \label{Nmc}
    \lim_{m\to \infty} \, \| \Phi (t_m, \theta_{-t_m} \om, g_{0,m} )\|_\Pi  = \|(\tilde{u} (\om), \tilde{v}(\om))\|_\Pi.
\end{equation}
Finally, for the Hilbert space $H^1 (\mbR) \times L^2 (\mbR)$, the weak convergence \eqref{twc} and the norm convergence \eqref{Nmc} imply the strong convergence. Therefore, up to a finite steps of subsequence selections always relabeled as the same in this proof, we reach the conclusion that
\begin{equation*}
    \Phi(t_m, \theta_{-t_m}\omega, g_{0,m} ) \rightarrow (\tilde{u}, \tilde{v}) \quad \textup{strongly in} \;\, H^1 (\mbR) \times L^2 (\mbR).
\end{equation*}
Thus the proof is completed.
\end{proof}

\section{\large{\bf{The Existence of Random Attractor}}}

In this section we shall first prove an instrumental convergence theorem in the space $L^p (X, \mathcal{M}, \mu)$ of Vitali type. It will pave the way to prove pullback asymptotic compactness of the cocycle $\Phi$ in the space $L^p (\mbR)$ for $2 < p < \infty$. This is the crcucial and final step to accomplish the proof of the existence of a random attractor for this random dynamical system $\Phi$ generated by the stochastic wave equation \eqref{1}-\eqref{2}.
\begin{theorem} \label{TVLp}
     Let $(X, \mathcal{M}, \mu)$ be a $\sigma$-finite measure space and assume that a sequence $\{f_m\}_{m=1}^\infty \subset L^p (X, \mathcal{M}, \mu)$ with $1 \leq p < \infty$ satisfies
\begin{equation} \label{ael}
    \lim_{m\to \infty} f_m (x) = f(x), \;\; \textup{a.e.}
\end{equation}
Then $f \in L^p (X, \mathcal{M}, \mu)$ and
\begin{equation} \label{fnf}
    \lim_{m \to \infty} \|f_m - f \|_{L^p (X, \mathcal{M},\, \mu)} = 0
\end{equation}
if and only if the following two conditions are satisfied\textup{:}

\textup{(a)} For any given $\varepsilon > 0$, there exists a set $A_\varepsilon \in \mathcal{M}$ such that $\mu (A_\varepsilon) < \infty$ and
\begin{equation} \label{conda}
    \int_{X \backslash A_\varepsilon} |f_m (x)|^p\, d\mu < \varepsilon, \quad \textup{for all} \;\, m \geq 1.
\end{equation}

\textup{(b)} The absolutely continuous property of the $L^p$ integrals of the functions in the sequence is satisfied uniformly,
\begin{equation} \label{condb}
    \lim_{\mu (Y) \to 0} \int_Y |f_m (x)|^p\, d\mu = 0, \quad \textup{uniformly in} \;\, m \geq 1.
\end{equation}
\end{theorem}
\begin{proof}
    First we prove the necessity. Statement (a): Under the condition \eqref{fnf}, for an arbitrarily given $\varepsilon > 0$ there exists an integer $N = N(\varepsilon) \geq 1$ such that
\begin{equation} \label{eN}
    \|f_m - f\|_{L^p (X, \mathcal{M},\, \mu)}^p < \frac{\varepsilon}{2^p}, \quad \textup{for all} \;\, m > N.
\end{equation}
Since $f \in L^p (X, \mathcal{M}, \mu)$, there exist measurable sets $B_\varepsilon$ and $S_\varepsilon$ both of finite measure, such that
\begin{equation} \label{BS}
    \int_{X \backslash B_\varepsilon} |f(x)|^p\, d\mu < \frac{\varepsilon}{2^p} \quad \textup{and} \quad \int_{X \backslash S_\varepsilon} |f_m (x)|^p\, d\mu < \varepsilon, \;\; \textup{for} \;\, m = 1, \cdots, N.
\end{equation}
Put $A_\varepsilon = B_\varepsilon \cup S_\varepsilon$. Then $\mu (A_\varepsilon) < \infty$ and we have
\begin{equation*} \label{aA}
    \begin{split}
    &\int_{X \backslash A_\varepsilon} |f_m (x)|^p\, d\mu = \int_{X \backslash A_\varepsilon} (|f_m (x) - f(x)| + |f(x)|)^p\, d\mu  \\
    \leq &\, 2^{p-1} \left(\int_{X} |f_m (x) - f(x)|^p\, d\mu + \int_{X \backslash B_\varepsilon} |f(x)|^p\, d\mu\right) < \frac{\varepsilon}{2} + \frac{\varepsilon}{2} = \varepsilon, \;\; \textup{for}\;\, m > N.
    \end{split}
\end{equation*}
Besides it follows from the second inequality in \eqref{BS} that
$$
    \int_{X \backslash A_\varepsilon} |f_m (x)|^p\, d\mu \leq \int_{X \backslash S_\varepsilon} |f_m (x)|^p\, d\mu < \varepsilon, \quad \textup{for} \;\, m = 1, \cdots , N.
$$
Therefore, the statement (a) is valid.

Statement (b): By the absolutely continuous property of Lebesgue
integral on a $\sigma$-finite measure space, for any given
$\varepsilon > 0$, there exists $\delta_0 = \delta_0 (\varepsilon) >
0$ such that whenever $\mu (Y) < \delta_0$ one has
\begin{equation} \label{AC}
    \int_Y |f(x)|^p \, d\mu < \frac{\varepsilon}{2^p} \quad \textup{and} \quad \int_Y |f_m (x)|^p\, d\mu < \varepsilon, \quad \textup{for} \;\, m = 1, \cdots, N,
\end{equation}
where $N = N(\varepsilon)$ is the same integer in \eqref{eN}. Then
for any measurable set $Y \subset X$ with $\mu (Y) < \delta_0$ one
also has
$$
    \int_Y |f_m (x)|^p\, d\mu \leq 2^{p-1} \left(\int_{X} |f_m (x) - f(x)|^p\, d\mu + \int_{Y} |f(x)|^p\, d\mu\right) < \varepsilon, \quad \textup{for} \;\, m > N.
$$
Thus the statement (b) is also valid.

Next we prove the sufficiency. Suppose the two conditions (a) and
(b) are satisfied. First of all, by the condition (a) and Fatou's
Lemma, for an arbitrarily given $\varepsilon > 0$ there exists a set
$A_\varepsilon$ of finite maesure with
$$
    \sup_{m\geq 1} \, \int_{X \backslash A_\varepsilon} |f_m (x)|^p\, d\mu < \varepsilon,
$$
which implies that the limit function $f$ in the assumption \eqref{ael} satisfies
\begin{equation} \label{Ftou}
     \int_{X \backslash A_\varepsilon} |f (x)|^p\, d\mu \leq \liminf_{m\to \infty} \int_{X \backslash A_\varepsilon} |f_m (x)|^p\, d\mu < \varepsilon.
\end{equation}
Hence it follows that
\begin{equation} \label{fint}
    f \in L^p (X \backslash A_\varepsilon)  \quad \textup{and} \quad \|f_m - f\|_{L^p (X \backslash A_\varepsilon)} < 2\varepsilon^{1/p}, \quad \textup{for all} \;\, m \geq 1.
\end{equation}
Therefore, the proof of $f \in L^p (X, \mathcal{M}, \mu)$ and \eqref{fnf} is reduced to proving that
\begin{equation} \label{fmfA}
    f \in L^p (Y) \quad \textup{and}  \quad \lim_{m \to \infty} \|f_m - f \|_{L^p (U)} = 0,
\end{equation}
for any given measurable set $Y \subset X$ with $\mu (Y) < \infty$.

Then by the condition (b), for any given $\varepsilon > 0$, there
exists $\delta_1 = \delta_1 (\varepsilon) > 0$ such that for any $S
\subset X$ with $\mu (S) < \delta_1$ one has
\begin{equation} \label{Sb}
    \int_S \, |f_m (x)|^p\, d\mu < \varepsilon^p, \quad \textup{uniformly in } \;\, m \geq 1.
\end{equation}
Consequently, by Fatou's lemma,
\begin{equation} \label{FS}
    \int_S \, |f (x)|^p\, d\mu \leq \liminf_{m\to\infty} \int_S \, |f_m (x)|^p\, d\mu < \varepsilon^p.
\end{equation}
By Egorov's theorem on Lebesgue integral over such a set $Y$ of
finite measure in the space $(X, \mathcal{M}, \mu)$, there exists a
measurable subset $B \subset Y$ with $\mu (Y \backslash B) <
\delta_1$ such that
\begin{equation*} 
    \lim_{m\to \infty} f_m (x) = f(x), \;\; \textup{uniformly  a.e. on} \; B,
\end{equation*}
so that there exists an integer $m_0 = m_0 (\varepsilon) \geq 1$ such that
\begin{equation} \label{m0}
    \|f_m - f\|_{L^p (B)} < \varepsilon, \quad \textup{for all} \;\, m \geq m_0.
\end{equation}
Combining \eqref{Sb}, \eqref{FS} and \eqref{m0}, we see that
\begin{equation*}
    \|f_m - f\|_{L^p (Y)} \leq \|f_m \|_{L^p (Y \backslash B)} + \|f \|_{L^p (Y \backslash B)} + \|f_m - f \|_{L^p (B)} < 3\varepsilon, \;\; \textup{for} \;\; m \geq m_0.
\end{equation*}
Therefore, \eqref{fmfA} is proved. The proof is completed.
\end{proof}

Finally we present and prove the main result of this work on the existence of a global attractor for this random dynamical system $\Phi$ on the product Banach space with arbitrary exponent and arbitrary space dimension.

\begin{theorem} \label{MT}
    Under the Assumptions \textup{I, II} and \textup{III}, the random dynamical system $\Phi$ generated by the damped stochastic wave equation \textup{(1.1)} on the Banach space $E= (H^1(\mathbb{R}^n) \cap L^p (\mathbb{R}^n)) \times L^2(\mathbb{R}^n)$ over the metric dynamical space $(\gw, \mathcal{F}, P, (\theta_t)_{t\in \mathbb{R}})$ has  a unique $\msd$-pullback random attractor $\mathcal{A} = \{\mathcal{A}(\omega)\}_{\om\in\gw} \in \msd$.
 \end{theorem}

\begin{proof}
Lemma \ref{L41} shows that there exists a $\msd$-pullback random absorbing set, the $K = \{B_E (0, R(\om))\}_{\om\in\gw}$ in the space $E$ for the cocycle $\Phi$. Thus it suffices to prove that the cocycle $\Phi$ is $\msd$-pullback asymptotically compact in $E$.

(1) Theorem \ref{T51} shows that for every $B=\{B(\omega)\}_{\omega\in \Omega}\in \msd$ and any sequence
$$
    \{\Phi (t_m, \theta_{-t_m} \om, g_{0,m})\}_{m=1}^\infty,
$$
where $t_m \to \infty$ and $g_{0,m} = (u_{0,m}, v_{0,m})\in B(\theta_{-t_m}\omega)$,
along a pullback quasi-trajectory of the cocycle $\Phi$ has a subsequence, which is denoted by the same, such that
\begin{equation} \label{ph2}
    \Phi(t_m, \theta_{-t_m} \omega, g_{0,m}) \longrightarrow (\tilde{u}(\om), \tilde{v}(\om)) \quad \textup{strongly in} \;\, H^1 (\mbR) \times L^2 (\mbR),
\end{equation}
and consequently
\begin{equation} \label{php}
    \mathbb{P}_u \, \Phi(t_m, \theta_{-t_m}\omega, g_{0,m}) \longrightarrow \tilde{u}(\om) \quad \textup{strongly in} \;\, L^2 (\mbR).
\end{equation}
Here $\mathbb{P}_u: (u, v) \mapsto u$ is the projection.

(2) Applying the first item in Lemma \ref{L51} to the space $L^2 (\mbR)$, it follows from \eqref{ph2} that there exists a subsequence $\{\Phi(t_{m_k}, \theta_{-t_{m_k}}\omega, g_{0, m_k})\}_{k=1}^\infty
$ of $\{\Phi(t_m, \theta_{-t_m}\omega, g_{0, m})\}_{m=1}^\infty$, such that
\begin{equation} \label{phae}
    \lim_{k\to\infty} \Phi(t_{m_k}, \theta_{-t_{m_k}}\omega, g_{0, m_k})(x) = (\tilde{u}(\om)(x), \tilde{v}(\om)(x)), \;\; \textup{a.e. in} \;\, \mbR.
\end{equation}
Hence we have
\begin{equation} \label{Puae}
     \lim_{k\to\infty} \mathbb{P}_u \Phi (t_{m_k}, \theta_{-t_{m_k}}\omega, g_{0, m_k}) (x) = \tilde{u}(\om)(x), \;\; \textup{a.e. in} \;\, \mbR.
\end{equation}
Therefore, the assumption \eqref{ael} in Theorem \ref{TVLp} is satisfied by the sequence of functions $\{\mathbb{P}_u \Phi (t_{m_k}, \theta_{-t_{m_k}}\omega, g_{0, m_k}) (x)\}_{k=1}^\infty$ in $L^p (\mbR)$.

(3) By Lemma \ref{L42}, for a.e. $\om \in \gw$ and any $\varepsilon > 0$, there exists an integer $k_0 = k_0 (B, \om, \varepsilon) > 0$ and $V = V(\om, \varepsilon) \geq 1$ such that for all $k > k_0$ one has
\begin{equation} \label{Va}
    \int_{\mbR \backslash B_V} |\mathbb{P}_u \Phi (t_{m_k}, \theta_{-t_{m_k}}\omega, g_{0, m_k}) (x)|^p \, dx \leq \interleave \Phi(t_{m_k}, \theta_{-t_{m_k}}\omega, g_{0, m_k}) \interleave^p_{E(\mbR \backslash B_V)}  < \varepsilon,
\end{equation}
for any $g_{0, m_k} \in B(\theta_{-t_{m_k}} \om)$, where $B_V$ is the ball centered at the origin with radius $V$ in $\mbR$. Then there exists $V_0 = V_0 (\om, \varepsilon) > 0$ such that
\begin{equation} \label{V0}
    \int_{\mbR \backslash B_{V_0}} |\mathbb{P}_u \Phi (t_{m_k}, \theta_{-t_{m_k}}\omega, g_{0, m_k}) (x)|^p \, dx < \varepsilon, \quad \textup{for} \;\, k = 1, \cdots, k_0.
\end{equation}
Thus \eqref{Va} and \eqref{V0} confirm that with $A_\varepsilon = B_{\max \{V, V_0\}}$ the condition (a) in Theorem \ref{TVLp} is satisfied by the sequence of functions $\{\mathbb{P}_u \Phi (t_{m_k}, \theta_{-t_{m_k}}\omega, g_{0, m_k}) (x)\}_{k=1}^\infty$ in $L^p (\mbR)$.

(4) Finally we prove that the uniform absolutely continuous condition (b) of Theorem \ref{TVLp} is also satisfied by the sequence of functions $\{\mathbb{P}_u \Phi (t_{m_k}, \theta_{-t_{m_k}}\omega, g_{0, m_k}) (x)\}_{k=1}^\infty$ in $L^p (\mbR)$. 

According to the Assumption II, for any measurable set $Y \subset \mbR$, we have
$$
    C_3 \int_Y |u|^p \, dx \leq \int_Y (F(x,u) + \phi_3 (x))\, dx \leq Q_Y (u, v),  \quad \textup{for} \;\, (u,v) \in E,
$$
where $Q_Y (u,v)$ is analogous to \eqref{511} and defined by

\begin{equation} \label{QY}
    Q_Y (u, v)=\|v\|_{L^2(Y)}^2+ \left(\alpha+\delta^2-\beta \delta\right)\|u\|_{L^2(Y)}^2 + \|\nabla u\|_{L^2(Y)}^2+2\int_{Y} (F(x, u) + \phi_3 (x))\, dx.
\end{equation}
We integrate the inequality \eqref{sgre} over the time interval $[-t_m, 0]$ to get
\begin{equation} \label {QYq}
    \begin{split}
    &\, Q_Y (u(0, \om, -t_m, u_{0,m}), v(0, \om, -t_m, v_{0,m})) \\
    \leq &\, e^{-\sigma t_m} Q_Y ((u_{0,m}, v_{0,m})) + \int_{-t_m}^0 e^{\sigma t} \left(2\Gamma_1^Y (\theta_t \om) + C_6 (Y) + \frac{\|g\|_{L^2 (Y)}^2}{\beta - \delta}\right) dt,
    \end{split}
\end{equation}
where, in view of \eqref{413} and the set-up of the constants $C_4$ and $C_6$ in Section 3.1, we have
\begin{equation*}
    \Gamma_1^Y (\theta_t \om) = C_0 \left(\|z(\ztw)\|_{L^2 (Y)}^2 + \|\nabla z(\ztw)\|_{L^2 (Y)}^2 + \|z(\ztw)\|_{L^p (Y)}^p\right)
\end{equation*}
and
\begin{equation} \label{C6Y}
    \begin{split}
    C_6(Y) &= 2\left(\delta C_2 - \frac{\varepsilon C_1 (p-1)}{C_3 p} \right)\|\phi_3 \|_{L^1 (Y)} + \varepsilon \|\phi_1 \|_{L^2 (Y)}^2  \\
    &\, - \delta \|\phi_2 \|_{L^1 (Y)} + \frac{\varepsilon C_1 (p-1)}{C_3 p} \|\phi_3 \|_{L^1 (Y)} \leq \varepsilon \|\phi_1 \|_{L^2 (Y)}^2 + 2\delta C_2 \|\phi_3 \|_{L^1 (Y)}.
\end{split}
\end{equation}
Note that $z(\ztw) = \sum_{j=1}^m h_j (x) z_j (\theta_t \om_j)$. By \eqref{39}, we obtain
\begin{equation} \label{G1Y}
    \begin{split}
    \Gamma_1^Y (\ztw) &= C \max_{1 \leq j \leq m} \left\{\|h_j \|_{H^1(Y)}^2, \|h_j \|_{L^p (Y)}^p\right\} \sum_{j=1}^m \left(|z_j (\theta_t \om_j)|^2 + |z_j (\theta_t \om_j )|^p \right) \\
    &\leq C \max_{1 \leq j \leq m} \left\{\|h_j \|_{H^1(Y)}^2, \|h_j \|_{L^p (Y)}^p\right\} e^{\frac{\sigma}{2}|t|} r_0 (\om),
    \end{split}
\end{equation}
where $C > 0$ is a constant.

Substitute the expression of $Q_Y ((u_{0,m}, v_{0,m}))$ for $(u_{0,m}, v_{0,m}) \in B(\theta_{-t_m} \om)$ and \eqref{C6Y}, \eqref{G1Y} into the inequality \eqref{QYq}. Since \eqref{33}-\eqref{34} yield
$$
    \int_Y (F(x, u) + \phi_3 (x))\, dx \leq \frac{1}{C_2} \left[C_1 \|u\|_{L^p (Y)}^p + \|u\|_{L^2 (Y)}^2 + \|\phi_1 \|_{L^2 (Y)}^2 + \|\phi_2 \|_{L^1 (Y)} \right],
$$
for every $\om \in \gw$ and $B \in \msd$ and any $g_{0,m} = (u_{0,m}, v_{0,m}) \in B(\theta_{-t_m} \om)$, we get
\begin{equation} \label{upY}
    \begin{split}
    &\, C_3 \int_Y |u(0, \om, -t_m, u_{0,m})|^p\, dx \leq Q_Y (u(0, \om, -t_m, u_{0,m}), v(0, \om, -t_m, v_{0,m})) \\
    \leq &\, e^{-\sigma t_m} \left[\|v_{0,m}\|_{L^2 (Y)}^2 + (\alpha + \delta^2 - \beta \delta)\|u_{0,m}\|_{L^2 (Y)}^2 + \|\nabla u_{0,m}\|_{L^2 (Y)}^2\right]  \\
    &\, + e^{-\sigma t_m}  \frac{1}{C_2} \left[C_1\|u_{0,m}\|_{L^p (Y)}^p + \|u_{0,m}\|_{L^2 (Y)}^2 + \|\phi_1 \|_{L^2(Y)}^2 + \|\phi_2 \|_{L^1 (Y)} \right]   \\
    &\, + \int_{-t_m}^0 2 e^{\sigma t} C \max_{1 \leq j \leq m} \left\{\|h_j \|_{H^1(Y)}^2, \|h_j \|_{L^p (Y)}^p\right\} e^{-\frac{\sigma}{2} t} r_0 (\om)\, dt   \\
    &\, + (\varepsilon + 2\delta C_2) \int_{-t_m}^0 e^{\sigma t} \left(\|\phi_1 \|_{L^2 (Y)}^2 + \|\phi_3 \|_{L^1 (Y)}\right) dt + \int_{-t_m}^0 \frac{e^{\sigma t}}{\beta - \delta} \|g\|_{L^2 (Y)}^2\, dt.
    \end{split}
\end{equation}
Due to the absolute continuity of the respective Lebesgue integrals of the functions $\phi_1 (x), \phi_2 (x), \phi_3 (x), h_j (x), j = 1, \cdots, m$, and $g$ involved in the above inequality \eqref{upY}, for an arbitrarily given $\eta > 0$, there exists $\mu_0 = \mu_0 (\om, \eta) > 0$ such that for any measurable set $Y \subset \mbR$ with $\mu (Y) < \mu_0$ one has
\begin{equation} \label{eta1}
    \begin{split}
    &\, e^{-\sigma t_m} \frac{1}{C_2} \left(\|\phi_1 \|_{L^2(Y)}^2 + \|\phi_2 \|_{L^1 (Y)} \right)   \\
    &\, + \int_{-t_m}^0 2 e^{\sigma t} C \max_{1 \leq j \leq m} \left\{\|h_j \|_{H^1(Y)}^2, \|h_j \|_{L^p (Y)}^p\right\} e^{-\frac{\sigma}{2} t} r_0 (\om)\, dt   \\
    &\, + (\varepsilon + 2\delta C_2) \int_{-t_m}^0 e^{\sigma t} \left(\|\phi_1 \|_{L^2 (Y)}^2 + \|\phi_3 \|_{L^1 (Y)}\right) dt + \int_{-t_m}^0 \frac{e^{\sigma t}}{\beta - \delta} \|g\|_{L^2 (Y)}^2\, dt  \\
    \leq &\, \frac{1}{C_2} \left(\|\phi_1 \|_{L^2(Y)}^2 + \|\phi_2 \|_{L^1 (Y)} \right) + \frac{4C}{\sigma} r_0 (\om) \max_{1 \leq j \leq m} \left\{\|h_j \|_{H^1(Y)}^2, \|h_j \|_{L^p (Y)}^p\right\}  \\
    &\, + \frac{1}{\sigma} (\varepsilon + 2\delta C_2) \left(\|\phi_1 \|_{L^2 (Y)}^2 + \|\phi_3 \|_{L^1 (Y)}\right) + \frac{1}{\sigma (\beta - \delta)} \|g\|_{L^2 (Y)}^2 < \frac{\eta}{2}.
    \end{split}
\end{equation}
Moreover, since it has been specified in the beginning of Section 3.1 that the universe $\msd = \msd_E$ and here $B \in \msd$, there exists a constant $C^* > 0$ such that
\begin{align*}
    &\, e^{-\sigma t_m} \left[\|v_{0,m}\|_{L^2 (Y)}^2 + (\alpha + \delta^2 - \beta \delta)\|u_{0,m}\|_{L^2 (Y)}^2 + \|\nabla u_{0,m}\|_{L^2 (Y)}^2\right]  \\
    &\, + \frac{1}{C_2} e^{-\sigma t_m} \left[C_1\|u_{0,m}\|_{L^p (Y)}^p + \|u_{0,m}\|_{L^2 (Y)}^2\right]  \\
    \leq &\, e^{-\sigma t_m} C^* \left(\|B(\theta_{-t_m} \om)\|_{E(Y)}^2 + \|B(\theta_{-t_m} \om)\|_{E(Y)}^p\right),
\end{align*}
where $\|B(\theta_{-t_m} \om)\|_{E(Y)} = \max_{g_0 \in B(\theta_{-t_m} \om)} \|g_0\, \zeta_Y \|_E$ with $\zeta_Y$ being the characteristic function for the set $Y$. Since $\lim_{t \to \-\infty} e^{-\sigma t} \|B(\theta_{-t} \om)\|_E = 0$, for the aforementioned arbitrary $\eta > 0$ there exists an integer $m_0 = m_0 (B, \om, \eta) \geq 1$ such that
\begin{equation} \label{eta2}
    \begin{split}
    &\, e^{-\sigma t_m} C^* \left(\|B(\theta_{-t_m} \om)\|_{E(Y)}^2 + \|B(\theta_{-t_m} \om)\|_{E(Y)}^p\right) \\
    \leq&\, e^{-\sigma t_m} C^* \left(\|B(\theta_{-t_m} \om)\|_E^2 + \|B(\theta_{-t_m} \om)\|_E^p\right) < \frac{\eta}{2}, \quad \textup{for all} \;\, m > m_0.
    \end{split}
\end{equation}
Then there esists $\mu_1 = \mu_1 (B, \om, m_0, \eta) > 0$ such that for any set $Y$ with $\mu (Y) < \mu_1$ one has
\begin{equation} \label{eta3}
     e^{-\sigma t_{j}} C^* \left(\|B(\theta_{-t_{j}} \om)\|_{E(Y)}^2 + \|B(\theta_{-t_{j}} \om)\|_{E(Y)}^p \right) < \frac{\eta}{2}, \quad  j = 1, \cdots, m_0.
\end{equation}

Put together \eqref{eta1}, \eqref{eta2} and \eqref{eta3} with \eqref{upY}. It shows that, for every $\om \in \gw$, whenever a measurable set $Y \subset \mbR$ satisfies $\mu (Y) < \min \{\mu_0, \mu_1\}$ one has
\begin{equation} \label{upcv}
     C_3 \int_Y |u(0, \om, -t_m, u_{0,m})|^p\, dx \leq  \frac{\eta}{2} + \frac{\eta}{2} = \eta, \quad \textup{for all} \;\, m \geq 1.
\end{equation}
Therefore,
\begin{equation} \label{Pufu}
    \lim_{\mu (Y) \to 0} \int_Y |\mathbb{P}_u \Phi (t_{m_k}, \theta_{-t_{m_k}} \om, g_{0,m})(x)|^p\, dx = 0, \quad \textup{uniformly in} \;\, k  \geq 1,
\end{equation}
so that the condition (b) of Theorem \ref{TVLp} is satisfied by the sequence of functions $\{\mathbb{P}_u \Phi (t_{m_k}, \theta_{-t_{m_k}} \om, g_{0,m_k})(x)\}_{k=1}^\infty$ in $L^p (\mbR)$.

As checked by the above steps (2), (3) and (4) in this proof, all the conditions in Theorem \ref{TVLp} are satisfied by the sequence of functions $\{\mathbb{P}_u \Phi (t_{m_k}, \theta_{-t_{m_k}}\omega, g_{0, m_k}) (x)\}_{k=1}^\infty$ in $L^p (\mbR)$. Then we apply Theorem \ref{TVLp} to obtain
\begin{equation} \label{psp}
     \lim_{k \to \infty} \mathbb{P}_u  \Phi(t_{m_k}, \theta_{-t_{m_k}}\omega, g_{0, m_k}) = \tilde{u}(\om), \;\; \textup{strongly in} \;\, L^p (\mbR).
\end{equation}

Finally, the combination of \eqref{ph2} and \eqref{psp} shows that there exists a convergent subsequence $\{\Phi(t_{m_k}, \theta_{-t_{m_k}}\omega, g_{0, m_k})\}_{k=1}^\infty$ of the sequence $\{\Phi (t_m, \theta_{-t_m} \om, g_{0,m})\}_{m=1}^\infty$ in the space $E = (H^1 (\mbR) \cap L^p (\mbR))\times L^2 (\mbR)$.
Therefore, the random dynamical system $\Phi$ on $E$ is $\msd$-pullback asymptotically compact.

According to Theorem \ref{T1}, we conclude that there exists a $\msd$-pullback random attractor $\mathcal{A} = \{\mathcal{A}(\omega) \}_{\om\in\gw} \in \msd$ for this random dynamical system $\Phi$ on $E$ generated by the stochastic wave equation \eqref{1}. The proof is completed.
\end{proof}

\end{document}